\documentclass[12pt]{amsart}
\usepackage{url}
\usepackage{amsthm}
\usepackage{amsmath}
\usepackage[margin=1in]{geometry}
\usepackage{amssymb,verbatim}

{\providecommand{\noopsort}[1]{} %year = "unpublished manuscript\setbox0=\hbox{2003}"

\theoremstyle{plain}
\newtheorem{theorem}{Theorem}[section]          \newcommand{\thm}[1]{Theorem \ref{#1}}
\newtheorem{corollary}[theorem]{Corollary}     \newtheorem*{nonumbercorollary}{Corollary}
\newtheorem{lemma}[theorem]{Lemma}              
\newtheorem{proposition}[theorem]{Proposition} 
\newtheorem*{nonumberconjecture}{Conjecture}

\newtheorem{claim}[theorem]{Claim}             
\newtheorem{main}{Theorem}
\theoremstyle{definition}
\newtheorem{definition}[theorem]{Definition}

\newtheorem{example}[theorem]{Example}

\numberwithin{table}{section}

\newcommand{\al}{\alpha}\newcommand{\be}{\beta}

\newcommand{\Z}{\mathbb{Z}}\newcommand{\Q}{\mathbb{Q}}

\newcommand{\s}{\mathbb{S}}

\newcommand{\SO}{\mathrm{SO}}
\newcommand{\Or}{\mathrm{O}}

\DeclareMathOperator{\sone}{S^1}
\DeclareMathOperator{\cod}{cod}

\DeclareMathOperator{\diag}{diag}

\DeclareMathOperator{\ave}{ave}

\newcommand{\of}[1]{\left(#1\right)}
\newcommand{\ceil}[1]{\left\lceil #1 \right\rceil}
\newcommand{\floor}[1]{\left\lfloor #1 \right\rfloor}
\newcommand{\pfrac}[2]{\left(\frac{#1}{#2}\right)}
\newcommand{\inner}[1]{\left\langle #1 \right\rangle}

\newcommand{\st}{~|~}

\newcommand{\id}{\mathrm{id}}

\author{Lee Kennard}
\title[Fundamental groups of manifolds with positive curvature and symmetry]{Fundamental groups of manifolds with positive sectional curvature and torus symmetry}
\address{Department of Mathematics, University of Oklahoma, Norman, Oklahoma 73019}
\email{kennard@math.ou.edu}
\urladdr{www.math.ou.edu/~kennard}

\begin{document}
\begin{abstract}
In 1965, S.-S. Chern posed a question concerning the extent to which fundamental groups of manifolds admitting positive sectional curvature look like spherical space form groups. The original question was answered in the negative by Shankar in 1998, but there are a number of positive results in the presence of symmetry. These classifications fall into categories according to the strength of their conclusions. We give an overview of these results in the case of torus symmetry and prove new results in each of these categories.
\end{abstract}

\maketitle

\vspace{-.2in}

\section{Introduction}\smallskip

The sphere is the simplest example of a compact manifold. With its standard Riemannian metric, it has positive sectional curvature. Strikingly, in all odd dimensions except for $7$ and $13$, the spheres remain the only simply connected, compact manifolds known to admit positively curved metrics. In view of this, it is not surprising that the realization problem for fundamental groups of positively curved manifolds remains open. As far as general obstructions go, one only has the classical results of Synge and Bonnet--Myers, as well as Gromov's universal bound on the number of generators (see \cite{Gromov78}).

Spherical space forms provide many examples of fundamental groups of positively curved manifolds (see Davis--Milgram \cite{DavisMilgram85}, Hambleton \cite{Hambleton15}, and Wolf \cite{Wolf}). However they form a very restrictive class of groups. In particular, they satisfy the properties that every abelian subgroup is cyclic and every involution is central. By analogy with Preissmann's theorem, S.-S. Chern asked whether the first of these properties holds for the fundamental groups of positively curved manifolds (see \cite[p.~167]{ChernProblem65}). The conjecture that the answer is yes was at times called the Chern conjecture (see Yau \cite[p. 671]{Yau82}, Petersen \cite{Petersen96}, and Berger \cite[p. 583]{Berger03}). 

Shankar \cite{Shankar98} answered Chern's question in the negative by showing that $\pi_1(M)$ can contain $\Z_2 \times \Z_2$, as well as dihedral groups of all orders. Further study and examples by Bazaikin \cite{Bazaikin99}, Grove--Shankar \cite{GroveShankar00}, and Grove--Shankar--Ziller \cite{GroveShankarZiller06} show that, in particular, $\pi_1(M)$ can contain $\Z_3\times\Z_3$. It remains unknown whether $\pi_1(M)$ can contain $\Z_p \times \Z_p$ for some $p \geq 5$.

These examples grew out of a larger program formulated by Grove whereby one restricts attention to positively curved manifolds with a large degree of symmetry. For example, Shankar's examples involve free actions on homogeneous spaces with positive curvature. For surveys on the interplay between positive curvature and symmetry in general, see Grove \cite{Grove02,Grove09}, Wilking \cite{Wilking07}, and Ziller \cite{Ziller07,Ziller14}.

On the other hand, restricting to the case of positively curved manifolds with symmetry leads to a number of fundamental group obstructions. It is the purpose of this article to examine such obstructions in the specific case of torus symmetry. 

\smallskip

The strongest possible obstructions for fundamental groups of positively curved manifolds imply that the groups are cyclic. These results are sharp in a sense since lens spaces $\s^{2m-1}/\Z_k$ have cyclic fundamental groups of arbitrary order and admit constant positive curvature metrics with $T^m$ symmetry. A number of papers show that this conclusion to hold for arbitrary positively curved manifolds under weaker symmetry assumptions. Two such results are highlighted here.

\begin{theorem}[Wilking, Frank--Rong--Wang]\label{thm:WFRW}
Let $M^n$ be a closed Riemannian manifold with positive sectional curvature and $T^r$ symmetry. The fundamental group is cyclic under each of the following conditions:
	\begin{itemize}
	\item $r \geq \frac{n+1}{4} + 1$.
	\item $r \geq \frac{n+1}{6} + 1$ and $n + 1 \not\equiv 0 \bmod{4}$.
	%\item $r \geq \frac{n+1}{6}$, $n+1 \not\equiv 0 \bmod{4}$, $n + 1 \not\equiv 0\bmod{6}$, and $n \geq 25$.
	%\item $r \geq \frac{n+1}{8} + 1$, $n \equiv 1 \bmod{4}$, $n \equiv 1 \bmod{6}$, and $n \geq 25$.
	%\item $r \geq \frac{n+1}{10} + 2$, $n \equiv 1 \bmod{8}$, $n \equiv 1 \bmod{6}$, and $n \geq 49$.
	\end{itemize}
%In particular, if $n$ is $25$, $37$, or $49$, then $\pi_1(M)$ is cyclic if $r$ is at least $5$, $6$, and $7$, respectively.
\end{theorem}

The first statement is proved in Wilking \cite[Theorem 4]{Wilking03} (cf. Rong \cite{Rong02}), where it is also shown to be sharp if $n + 1 \equiv 0 \bmod{4}$. The second statement is proved in Frank--Rong--Wang \cite[Theorem A]{FrankRongWang13}, and it is shown there to be sharp if $n + 1 \equiv 0 \bmod{6}$. Further improvements exist in some sufficiently large dimensions satisfying the properties that $n+1 \not\equiv 0 \bmod{4}$ and $n+1 \not\equiv 0 \bmod{6}$ (see Rong--Wang \cite{RongWang08} and Wang \cite[Theorem B]{Wang07}), but these results are not known to be sharp.

In Example \ref{exa:SCC}, we construct a family of examples of space forms with large torus actions and non-cyclic fundamental groups. When $n + 1 \equiv 0 \bmod{4}$ or $n + 1 \equiv 0 \bmod{6}$, these examples are those of Wilking and Frank--Rong--Wang. In general, when $q$ is the smallest prime such that $n + 1 \equiv 0 \bmod{2q}$, there exists a non-cyclic group $\Gamma$ and a space form $\s^n/\Gamma$ that admits $T^r$ symmetry with $r = \frac{n+1}{2q}$. Therefore, the best possible extension of Theorem \ref{thm:WFRW} would be the following:

\begin{nonumberconjecture}
Let $q$ be a prime. Assume $n$ is odd such that $q$ is the smallest prime dividing $\frac{n+1}{2}$. If $M^n$ be a closed Riemannian manifold with positive sectional curvature and $T^r$ symmetry such that $r \geq \frac{n+1}{2q} + 1$, then $\pi_1(M)$ is cyclic.
\end{nonumberconjecture}

When $q = 2$ or $q = 3$, the conjecture holds by Theorem \ref{thm:WFRW}.  An equivalent version of this conjecture is stated in Wang \cite{Wang07}, and it is proved in the case of constant sectional curvature (see \cite[Theorem A]{Wang07}). Our first main theorem generalizes Wang's result by proving the conjecture in the case where the universal cover is a homotopy sphere. It also proves the conjecture for all $q$ in all but finitely many dimensions, under the additional assumption that the torus action has no fixed point.

\begin{main}\label{thm:SCC}
The conjecture above holds in each of the following cases:
	\begin{enumerate}
	\item\label{partSCCintegral} The universal cover of $M$ is homeomorphic to a sphere.
	\item\label{partSCCempty} The dimension of $M$ is at least $16q^2$, and the torus action has no fixed point.
	\end{enumerate}
\end{main}

Moving beyond cyclic fundamental groups, the next important class of groups are those that act freely and linearly on $\s^3$. These are called the three-dimensional spherical space form groups, and they are classified (see Wolf \cite[Section 7.4]{Wolf}). There are again a number of results that provide conditions under which a positively curved manifold with torus symmetry has fundamental group isomorphic to a three-dimensional spherical space form group. Here are two (see \cite[Theorem A]{FrankRongWang13} and \cite[Theorem 1.1]{RongWang08}).

\begin{theorem}[Frank--Rong--Wang, Rong--Wang]\label{thm:FRWRW}
If $M^{n}$ is a closed Riemannian manifold with positive sectional curvature and $T^r$ symmetry, then $\pi_1(M)$ is the fundamental group of a three-dimensional spherical space form under each of the following conditions:
	\begin{itemize}
	\item $r \geq \frac{n+1}{6} + 1$.
	\item $r \geq \frac{n+1}{6}$, $n + 1 \not\equiv 0 \bmod{6}$, and $n \geq 25$.
	\end{itemize}
\end{theorem}

If $r \geq \frac{n+1}{6}$, $n + 1 \equiv 0 \bmod{6}$, and $n \geq 25$, Rong and Wang \cite{RongWang08} prove a partial classification that either $\pi_1(M)$ is a three-dimensional spherical space form group or $\pi_1(M)$ acts freely on a positively curved rational (homology) $5$--sphere. The second main result of this paper sharpens the latter possibility in two ways.

\begin{main}\label{thm:RongWang5mod6}
Let $M^{n}$ be a closed Riemannian manifold with positive sectional curvature and $T^r$ symmetry. If $n \geq 25$ and $r \geq \frac{n+1}{6}$, then either $\pi_1(M)$ is a three-dimensional spherical space form group or $r = \frac{n+1}{6}$ and $\pi_1(M)$ acts freely both by diffeomorphisms on $\s^5$ and by isometries on a positively curved mod $2$ and mod $3$ homology $5$--sphere.
\end{main}

We remark that the action by diffeomorphisms on $\s^5$ might not be linear, so we would like to underline Rong and Wang's question in \cite{RongWang08} about whether $\pi_1(M)$ is a five-dimensional spherical space form group in this case.

\smallskip

Moving beyond spherical space form groups, there is the larger class of finite groups that act freely and smoothly, but not necessarily linearly, on standard spheres. According to the classification of Madsen, Thomas, and Wall, these are precisely the finite groups that satisfy the two following properties:
	\begin{itemize}
	\item ($p^2$ conditions) $\pi_1(M)$ does not contain a copy of $\Z_p \times \Z_p$ for any prime $p$.
	\item ($2p$ conditions) every involution in $\pi_1(M)$ is central.
	\end{itemize}
The first of these is equivalent to the property that every abelian subgroup of $\pi_1(M)$ is cyclic, as in Chern's original question. The second is proved in Milnor \cite{Milnor57} for finite groups that act freely on a manifold that is also a mod $2$ homology sphere (cf. Bredon \cite[Chapter III.8.3]{Bredon72} and Davis \cite[Corollary 7.5]{Davis83}). One might therefore try to prove that these properties carry over to fundamental groups of positively curved manifolds with large symmetry. With these goals in mind, Sun and Wang nearly proved both the $p^2$ and $2p$ conditions under a symmetry assumption weaker than in the results above (see \cite[Theorems A and B]{SunWang09}):

\begin{theorem}[Sun--Wang]\label{thm:SunWang}
If $M^n$ is a closed Riemannian manifold with positive sectional curvature and $T^r$ symmetry such that $n \geq 23$ and $r \geq \frac{n+3}{8} + 1$, then $\pi_1(M)$ does not contain $\Z_p \times \Z_p$ for any prime $p\neq 3$, and every involution in $\pi_1(M)$ is central.
\end{theorem}

As an application of the results in this article, we can remove the condition that $p \neq 3$ in this theorem. In fact, we require less symmetry to do so in the case where $\dim(M) \equiv 1 \bmod{4}$.

\begin{main}\label{thm:CCS12}
Let $M^n$ be a closed Riemannian manifold with positive sectional curvature.
	\begin{enumerate}
	\item If $r \geq \frac{n+3}{8} + 1$ and $n \geq 23$, then $\pi_1(M)$ acts freely and smoothly on a sphere.
	\item If $r \geq \frac{n-1}{12} + 3$ and $n \equiv 1 \bmod{4}$, then $\pi_1(M)$ acts freely and smoothly on $\s^5$ or $\s^9$.
	\end{enumerate}
In either case, the fundamental group satisfies all $p^2$ and $2p$ conditions.
\end{main}

Part (2) of Theorem \ref{thm:CCS12} is actually one in a family of results where the symmetry assumption is $r \geq \frac{n-1}{4q} + q$ for some positive number $q$ and where part of the conclusion is that $\pi_1(M) \supseteq\Z_p \times \Z_p$ for some prime $p$ only if $2 < p < q$ (see Theorem \ref{thm:CCS4q}). Furthermore, this family of results follows from the next theorem, which is the main technical result of this paper.

\begin{main}\label{thm:CCSlog}
Let $M^n$ be a closed, positively curved Riemannian manifold with $n \equiv 1 \bmod{4}$ and $T^r$ symmetry. If $r > \log_{4/3}\pfrac{n+3}{6}$, then $\pi_1(M) \cong \Z_{2^e} \times \Gamma$ for some $e\geq 0$ and some group $\Gamma$ of odd order. Moreover, the following hold:
	\begin{enumerate}
	\item If $\Gamma \supseteq \Z_p \times \Z_p$ for some prime $p$, then $r \leq \frac{n+1}{2p}$ and $r \leq \frac{n+1}{4p} + \frac{p}{2}$.
	\item If $\Gamma \not\supseteq \Z_p \times \Z_p$ for all $p$, then there exists an odd integer $d \leq \frac{n+1}{2r}$ such that $\pi_1(M)$ contains a normal, cyclic subgroup of index $d$ and acts freely and smoothly on $\s^{2d-1}$.
	\end{enumerate}
\end{main}

The factorization of $\pi_1(M)$ into a cyclic $2$--group and an odd-order group immediately implies that every involution is central. This factorization follows from a general surgery-theoretic result of Davis and Weinberger, applied to the special case of a positively curved, rational (homology) sphere of dimension $n \equiv 1 \bmod{4}$. For further discussion of their result and about the structure of $\pi_1(M)$ in the context of this theorem, see Section \ref{sec:DavisWeinbergerBurnside}. 

\smallskip

Finally, continuing along the path of results with weaker symmetry assumptions and correspodingly weaker fundamental group classifications, we come to a major result of Rong. It requires only a circle action (see \cite{Rong99}, cf. \cite[Theorem B]{Rong05}):

\begin{theorem}[Rong]
For each $n$, there exists a constant $w(n)$ such that if $M^n$ is a closed Riemannian manifold with positive sectional curvature and circle symmetry, then $\pi_1(M)$ has a cyclic subgroup of index at most $w(n)$.

In particular, if $\pi_1(M) \supseteq \Z_p \times \Z_p$ for some prime $p$, then $p \leq w(n)$.
\end{theorem}

The formula for $w(n)$ grows super-exponentially in $n$, however examples on space forms suggest that it might be replaced by a linear function $n$. The next result provides evidence for this suggestion in the special case where the universal cover is a rational sphere.

\begin{nonumbercorollary}[Corollary \ref{cor:RongRationalSphere}]
Let $M^n$ be a closed, positively curved Riemannian manifold that admits a nontrivial, isometric circle action. If the universal cover of $M$ is a rational sphere and $\pi_1(M) \supseteq \Z_p\times\Z_p$ for some prime $p$, then $p \leq \frac{n+1}{2}$.
\end{nonumbercorollary}

This result is an immediate consequence of one of the new obstructions proved in this article (see Proposition \ref{pro:CircleObstruction} and the discussion that follows).

In addition to the standard sphere, there are two simply connected rational $7$--spheres known to admit positive curvature, the Berger space $\SO(5)/\SO(3)$ (see \cite{Berger61}) and a cohomogeneity one manifold $P_2$ homeomorphic but not diffeomorphic to the unit tangent bundle of $\s^4$ (see Dearricott \cite{Dearricott11} and Grove--Verdiani--Ziller \cite{GroveVerdianiZiller11}). In addition, there is an infinite list of rational $7$--spheres that are considered candidates to admit positive curvature (see Grove--Wilking--Ziller \cite{GroveWilkingZiller08} and Verdiani--Ziller \cite{VerdianiZiller14}), and Petersen and Wilhelm claim that the Gromoll--Meyer exotic $7$--sphere admits positive curvature (see \cite{PetersenWilhelm09pre}, cf. Joachim--Wraith \cite[Section 2]{JoachimWraith08} and references therein). With these examples in mind, we state the following strengthening of this corollary in dimension seven.

\begin{nonumbercorollary}[Corollary \ref{cor:dim7}]
Let $M^7$ be a closed, positively curved Riemannian manifold whose universal cover is a rational sphere. If a circle acts effectively by isometries on $M$, then every subgroup of $\pi_1(M)$ of odd order is cyclic.
\end{nonumbercorollary}

\smallskip
The proofs in this article require tools of two types. The first collection of tools provide cohomological information about positively curved manifolds with torus symmetry. Most significantly, we rely on ideas developed in Wilking \cite{Wilking03} and refined by the author and Amann (see \cite{AmannKennard14,AK3,Kennard13,Kennard14}). They include the connectedness lemma, results concerning periodicity in cohomology, and a powerful machinery for inducting over dimension. In this article, these results serve to essentially reduce the proofs of the main theorems to the case where the universal cover of the manifold $M$ is a rational sphere. This is one of the places where the assumption of $\dim(M) \equiv 1 \bmod{4}$ is used.

The second collection of tools for the proofs are three new obstructions to free actions of finite groups on manifolds with special cohomology, curvature, or symmetry properties. The first extends a result of Smith that obstructs free actions by $\Z_p \times \Z_p$ on mod $p$ homology spheres. This involves a technical lemma which refines the periodicity theorem in earlier work \cite{Kennard13} of the author (see Proposition \ref{pro:MorePeriodicity}). The proof of this refinement builds upon the use of Steenrod powers in \cite{Kennard13} and uses in addition secondary cohomology operations. The second obstruction is an immediate application of a surgery-theoretic result of Davis and Weinberger on free, orientation-preserving actions on rational spheres $M$ with $\dim(M) \equiv 1\bmod{4}$ (see Proposition \ref{pro:DavisWeinberger}).  This obstruction essentially reduces all of our proofs to the case of odd-order fundamental groups. For such groups, we are able to apply a classification Burnside, results about finite groups with periodic cohomology, and a realization theorem due to Madsen, Thomas, and Wall. The third obstruction was inspired by, and generalizes, an obstruction used in earlier work of Rong, Sun, and Wang. It provides an obstruction to free, isometric actions by $\Z_p\times\Z_p$ on positively curved rational spheres that commute with an isometric circle action (see Proposition \ref{pro:CircleObstruction}).

\smallskip
The outline of this article is as follows. Section \ref{sec:Preliminaries} states conditions under which a simply connected, positively curved manifold with torus symmetry is a rational sphere. Sections \ref{sec:MorePeriodicity}, \ref{sec:DavisWeinbergerBurnside}, and \ref{sec:CircleObstruction} provide proofs of the three new obstructions mentioned above. Section \ref{sec:DavisWeinbergerBurnside} also collects results about finite groups with periodic cohomology. Section \ref{sec:Room} pulls together these obstructions and applies them to torus actions on positively curved rational spheres. Theorem \ref{thm:RongWang5mod6} is also proved in this section. Theorem \ref{thm:CCSlog} is proved in Section \ref{sec:CCSlog} using the results of Sections \ref{sec:Preliminaries} and \ref{sec:Room}, and then it is used to prove Theorems \ref{thm:CCS12} and \ref{thm:SCC} in Sections \ref{sec:CCS} and \ref{sec:SCC}, respectively.

\vspace{.1in}
   
\noindent\textbf{Acknowledgements:} 
This work began as part of the author's Ph.D. thesis, and he is pleased to thank his advisor, Wolfgang Ziller, for multiple discussions and comments on ealier drafts. The author is also grateful to Daryl Cooper and Darren Long for directing him to the work of Jim Davis and Shmuel Weinberger, and to Jim Davis for a discussion of this work. The author is grateful for support from NSF Grants DMS-1045292 and DMS-1404670.

\tableofcontents

\smallskip
\section{Preliminaries}\label{sec:Preliminaries}
\bigskip

In this section, we mention a few of preliminary results and tools that have been used repeatedly to study fundamental groups of positively curved manifolds with torus symmetry. In addition, we state three important sufficient conditions under which a positively curved manifold with torus symmetry is a rational homology sphere.

The first crucial result is the following (see Berger \cite{Berger66} and Grove--Searle \cite{GroveSearle94}).

\begin{theorem}[Berger]\label{thm:Berger} 
If $T^r$ acts isometrically on a compact, positively curved manifold $M^n$, then there exists a codimension one subtorus $T^{r-1}$ that fixes some point of $M$.
\end{theorem}

This result ensures the existence of many non-trivial fixed-point sets in the presence of isometric torus actions on manifolds with positive sectional curvature. Such fixed-point sets are (embedded) totally geodesic submanifolds, so the classical result of Frankel comes into play (see \cite{Frankel61}). Wilking proved a vast generalization of Frankel's result, and we reproduce part of it here (see \cite[Theorem 2.1]{Wilking03}).

\begin{theorem}[Connectedness lemma]\label{thm:Connectedness}
Let $M^n$ be a closed, positively curved Riemannian manifold.
	\begin{enumerate}
	\item If $N^{n-k}$ is a closed, totally geodesic submanifold, then the inclusion is $(n-2k+1)$--connected.
	\item If $N_i^{n-k_i} \to M^n$ are two closed, totally geodesic submanifolds with $k_1 \leq k_2$, then the inclusion $N_1 \cap N_2 \to N_2$ is $(n - k_1 - k_2)$--connected.
	\end{enumerate}
\end{theorem}

We recall that an inclusion $N \to M$ is $c$--connected if the induced map on homotopy is isomorphisms in degrees less than $c$ and a surjection in degree $c$. In particular, we have the following consequences, which we will use repeatedly:
	\begin{itemize}
	\item If $\cod(N) \leq \frac{n-1}{2}$ in the first part of the lemma, then $\pi_1(N) \cong \pi_1(M)$.
	\item Suppose $N_i \subseteq M^{\iota_i}$ are fixed-point components containing $x$ of isometries $\iota_i$. If $n - k_1 - k_2 \geq 1$, then the intersection $N_1 \cap N_2$ is connected and hence coincides with the component containing $x$ of the fixed-point set $M^{\inner{\iota_1,\iota_2}}$ of the fixed-point set of the group generated by $\iota_1$ and $\iota_2$.
	\end{itemize}

We proceed to a discussion of three important conditions under which a simply connected positively curved Riemannian manifold is a rational homology sphere. The first is the following, which we quote frequently (for a proof, see \cite[Proposition 3.2]{AmannKennard14}):

\begin{lemma}[dk $2$ lemma]\label{lem:dk2}
Let $M^n$ be a closed, simply connected, positively curved Riemannian manifold with $n \equiv 1 \bmod{4}$, let $T$ be a torus acting isometrically and effectively on $M$, and let $T' \subseteq T$ denote a subtorus that fixes a point $x \in M$.

If there exists an involution $\sigma \in T'$ such that the component $N\subseteq M^\sigma$ containing $x$ has $\dim(N) \geq \frac{n}{2}$ and is fixed by another involution in $T'$, then $M$ is a rational sphere. In particular, this applies if the dimension of the kernel of the induced $T'$--action on $N$ is at least two.
\end{lemma}

The second and third sufficient conditions can also be regarded as statements about the case where $M$ is not a rational homology sphere. The first has a strong symmetry assumption and implies that every involution in the torus has large codimension. The second has a weak symmetry assumption and implies that some involution has large codimension.

\begin{proposition}\label{pro:AKsqrt}
Let $M^n$ be a closed, simply connected, positively curved Riemannian manifold with $n \equiv 1 \bmod{4}$ and $T^r$ symmetry such that $r \geq 2 \log_2 n + \frac{k}{2} + 1$ for some positive $k \leq \frac{n+2}{4}$. One of the following occurs:
	\begin{itemize}
	\item $M$ is a rational homology sphere.
	\item Every involution in $T^r$ has fixed-point set of codimension greater than $k$.
	\end{itemize}
In particular, if $n > 81$ and $r \geq 2 \sqrt{n}$, then either $M$ is a rational sphere or every involution has fixed point set of codimension greater than $\sqrt{n}$.
\end{proposition}

The first part of statement is simply a rephrasing of Proposition 3.3 in \cite{AmannKennard14}. The last statement follows immediately by taking $k = \sqrt{n}$ and confirming the estimate
	\[2 \sqrt{n} \geq 2\log_2 n + \frac{\sqrt{n}}{2} + 1\]
for all $n \geq 85$.

The third sufficient condition is the following. It is a slight adaptation of Proposition 1.1 in \cite{AK3} that is better suited to our situation.

\begin{proposition}\label{pro:AKlog}
Assume $n\equiv 1\bmod{4}$ and $n\geq 25$. Let $M^n$ be a closed, simply connected Riemannian manifold with positive sectional curvature, and let $T^s$ be a torus acting effectively by isometries on $M$ with fixed point $x$. If $s > \log_{4/3}\of{\frac{n+3}{6}} - 2$, then one of the following occurs:
\begin{itemize}
\item $M$ is a rational homology sphere.
\item There exists an involution $\iota \in T^s$ such that the component $N \subseteq M^\iota$ of the fixed-point set containing $x$ satisfies $\frac{n-1}{4} < \cod(N) \leq \frac{n-1}{2}$, and $\dim\ker(T^s|_N) \leq 1$.
\end{itemize}
\end{proposition}

Here, and throughout this article, if $T$ is a torus acting on a manifold $M$, and if $N \subseteq M$ is a submanifold to which the $T$--action restricts, then $\dim\ker(T|_N)$ denotes the dimension of the kernel of the induced $T$--action on $N$. Note that $N$ admits an isometric torus action of rank at least $\dim(T) - \dim\ker(T|_N)$.

\begin{proof}
First, it is immediate from the definition of periodic cohomology (see \cite[Definition 1.8]{Kennard13}) that $M$ is a rational homology sphere if and only if it has $4$--periodic rational cohomology.

The proof now parallels the proof of Proposition 1.1 in \cite{AK3}. The slight modification is that, here, we set $c = 1$, $t = 1$, $k_0  = \frac{n-1}{4}$, and $j = \floor{\log_2(k_0)}$. As there, the proof now follows immediately from Lemmas 1.2 and 1.3 in \cite{AK3} assuming a certain inequality involving the rank of the torus and the dimension of the manifold. We are referring to Inequality (1.1) in \cite{AK3}, reproduced here for $c = 1$, $t = 1$, and $n \equiv 1 \bmod{4}$:
\begin{equation}\label{ine:1.1}
\frac{n-1}{2} < j - 1 + \sum_{i=0}^{s - j} \ceil{\frac{n+3}{2^{i+2}}}.
\end{equation}

Dropping the ceiling functions in the proposed estimate and rearranging, we see that this inequality holds if
	\[s > \log_2(n+3) + j - \log_2(j+1) - 2.\]
One can check by hand that this holds for $25 \leq n \leq 45$. In addition, since $j \leq \log_2\pfrac{n-1}{4}$ and $j > \log_2\pfrac{n-1}{4} - 1$, Inequality (\ref{ine:1.1}) holds if
	\[s > \log_2(n+3) + \log_2(n-1) - \log_2(\log_2(n-1) - 2) - 4.\]
This inequality holds for $n = 49$ and hence for all $n \geq 49$ since the left-hand side grows faster than the right-hand side.
\end{proof}

\bigskip\section{Free actions on mod $p$ spheres}\label{sec:MorePeriodicity}\bigskip

In the classification of spherical space forms, it is shown that a finite group acts freely on sphere only if every abelian subgroup is cyclic. Smith proved a localized version of this, which we state here for manifolds (also see Bredon \cite[III.8.1]{Bredon72}):

\begin{theorem}[Smith, \cite{Smith44}]\label{thm:Smith44}
If $M$ is a mod $p$ homology sphere, then $\Z_p \times \Z_p$ does not act freely on $M$.
\end{theorem}

Here and throughout the paper, a mod $p$ (resp. rational) homology sphere is a smooth manifold whose mod $p$ (resp. rational) homology is that  of a sphere. We will apply Smith's theorem to conclude the following, the main result of this section:

\begin{proposition}\label{pro:MorePeriodicity}
Let $p\geq 3$ be prime, and let $n$ be an odd integer greater than $\max(19,2p^2)$. If $M^n$ is a closed,  simply connected Riemannian manifold with positive sectional curvature that contains two transversely intersecting, totally geodesic submanifolds of codimension $2p$, then $M$ is a mod~$p$ homology sphere. In particular, $M$ does not admit a free action by $\Z_p \times \Z_p$.
\end{proposition}

This result can be viewed as a refinement of the periodicity theorem due to the author (see \cite[Theorem B]{Kennard13}, and more specifically \cite[Proposition 2.1]{Kennard13}). That theorem immediately implies that the manifold is a rational homology sphere. To exclude free actions by $\Z_p \times \Z_p$, however, we need to know that the manifold is also a mod $p$ homology sphere.

We remark that Proposition \ref{pro:MorePeriodicity} is enough to exclude the possibility of $\pi_1(M) \supseteq \Z_3 \times \Z_3$ in Theorem \ref{thm:SunWang} of Sun and Wang in the introduction. However we postpone discussion of this until later, since the other obstructions we prove allow us to say more in this setting.

\begin{proof}
Note that the last conclusion is simply Smith's theorem, so it suffices to prove the first conclusion. It is straightforward to see that the assumptions on $p$ and $n$ imply that $2p \leq \frac{n+3}{4}$. Given this, we proceed as in the proof of the periodicity theorem in \cite{Kennard13} (see Definition 1.1 and Section 4 in \cite{Kennard13}). By the connectedness lemma, the existence of the two transversely intersecting, totally geodesic submanifolds of small codimension implies that $H^*(M;\Z_p)$ is $(2p)$--periodic. The theorem now follows from Lemma \ref{lem:MorePeriodicityAlg} below.
\end{proof}

\begin{lemma}\label{lem:MorePeriodicityAlg}
Let $M^n$ be a closed, simply connected manifold of odd dimension. Let $p$ be an odd prime such that $2p^2 \leq n$. If $H^*(M;\Z_p)$ is $(2p)$--periodic, then $M$ is a mod $p$  sphere.
\end{lemma}

This result refines Proposition 2.1 in \cite{Kennard13}, the proof of which uses mod $p$ Steenrod cohomology operations and applies not just to manifolds. Here we restrict to manifolds since the proof of Lemma \ref{lem:MorePeriodicityAlg} uses Poincar\'e duality. In addition, we require a {\it secondary factorization} by secondary cohomology operations due to Liulevicius and Shimada--Yamanoshita. After setting up the notation and some preliminaries, we prove the lemma.

First, we recall the existence of the Steenrod powers
	\[P^i:H^*(M;\Z_p) \longrightarrow H^{* + 2(p-1)i}(M;\Z_p)\]
and the Bockstein homomorphism
	\[\beta : H^*(M;\Z_p) \longrightarrow H^{* + 1}(M;\Z_p)\]
associated to the short exact sequence
	\[0 \longrightarrow \Z_2 \longrightarrow \Z_4 \longrightarrow \Z_2 \longrightarrow 0.\]
See \cite[Section 2]{Kennard13} or Hatcher \cite[Section 4.L]{Hatcher01} for a summary of the basic properties of these operations.

Second, we require a {secondary decomposition} of Steenrod powers. Such decompositions of Steenrod squares were developed by Adams in his resolution of the Hopf invariant one problem (see \cite{Adams60}). Liulevicius, Shimada, and Yamanoshita developed odd prime analogues for the Steenrod powers (see  \cite{Liulevicius62,ShimadaYamanoshita61}). We only require here the following consequence (see Harper \cite[Theorem 6.2.1.b]{Harper02}):

\begin{theorem}[A secondary decomposition of Steenrod powers]\label{thm:AdamsFactorization}
For any space $M$ and $x \in H^k(M;\Z_p)$, if $\beta(x) = 0$ and $P^1(x) = 0$, then there exist cohomology elements $w_0,w_1\in H^*(M;\Z_p)$ such that
	\[P^{p}(x) = \beta(w_0) + P^1(w_1).\]
\end{theorem}

We spend the rest of this section on the proof of Lemma \ref{lem:MorePeriodicityAlg}. We will denote $2p$ by $k$, $H^i(M;\Z_p)$ by $H^i$, and $H^*(M;\Z_p)$ by $H^*$.

Let $x \in H^k$ be an element inducing periodicity in $H^*$. First note that, if $x = 0$, then the conclusion is trivial. Second, since $k = 2p$, if another element induces periodicity with degree less than $k$, then $H^*$ is $2$--periodic by a proof similar to that of \cite[Lemma 3.2]{Kennard13}. Since $M$ is simply connected and odd-dimensional, this would imply that $M$ is a mod $p$ sphere. 

Suppose therefore that $x$ is nonzero and has minimal degree among all elements inducing periodicity. In particular, $x^p \neq 0$, $k$ does not divide $n-1$, and, by \cite[Lemma 2.3]{Kennard13}, no nontrivial multiple of $x$ is of the form $P^i(y)$ with $i > 0$. Similarly, we have the following:

\begin{lemma}
No nontrivial multiple of $x$ is of the form $\be(y)$.
\end{lemma}
\begin{proof}
Indeed, if such an expression exists, the Adem relation
	\[P^p \be = P^1 \be P^{p-1} + \be P^p\]
would imply
	\[x^p = P^p(x) = P^p\be(y) = P^1 \be P^{p-1}(y) + \be P^p(y).\]
But the first term on the right-hand side vanishes by \cite[Lemma 2.4]{Kennard13}, and the second term vanishes since $2p > \deg(y)$, so $x^p = 0$. This would contradict our assumption that $x^p \neq 0$.
\end{proof}

Next, choose $m \in \{2,4,6,\ldots,k-2\}$ such that $k$ divides $n-1-m$. We prove the following:
\begin{lemma}
The groups $H^m$, $H^{m+2}$,\ldots, $H^{k-2}$ vanish.
\end{lemma}
\begin{proof}
First, the definition of periodic cohomology implies that 
	\[H^m \cong H^{m + k} \cong H^{m+2k} \cong \cdots \cong H^{n-1}\] 
since $k$ divides $n-1-m$. By Poincar\'e duality and the assumption that $M$ is simply connected, we conclude that $H^m = 0$. We proceed by induction. Suppose that $H^l = 0$ for some $l\in\{m, m+2, \ldots, k-4\}$, and suppose that $H^{l+2} \neq 0$. Using Poincar\'e duality and periodicity, we conclude the existence of a relation
	\[u v = xz\]
for some $u \in H^{l+2}$ and $v \in H^{k+m+1-(l+2)}$, where $z \in H^{m+1} \cong H^k \cong \Z_p$ is a generator. On one hand, $P^1(u) \in H^{k+l} \cong H^l = 0$, so the Adem relations imply
	\[i! P^i(u) = P^1P^1\cdots P^1(u) = 0\]
and hence $P^i(u) = 0$ for all $0 < i < p$. Applying the Cartan relation and the property that $P^p(u) = 0$ and $P^p(v) = 0$ since $2p > \deg(u)$ and $2p > \deg(v)$, respectively, we conclude
	\[P^p(uv) = P^p(u)v + \sum_{0<i<p} P^i(u) P^{p-i}(v) + uP^p(v) = 0.\]

On the other hand, we claim that $P^p(xz) = x^p z$, which is nonzero by periodicity. Indeed, $P^p(x)z = x^p z$ and $xP^p(z) = 0$ since $2p = \deg(x)$ and $2p > \deg(z)$. In addition, if we can show that $P^1(z) = 0$, it will follow from the relation $j! P^j(z) = P^1P^1\cdots P^1(z)$ that $P^j(z) = 0$ for all $0 < j < p$, and hence that
	\[P^p(xz) = x^p z + \sum_{0 < j < p} P^{p-j}(x) P^j(z) + x P^p(z) = x^p z.\]
It suffices to prove $P^1(z) = 0$. If it is nonzero, Poincar\'e duality and periodicity imply the existence of $w \in H^2$ such that $w P^1(z) = xz$. Applying the Cartan formula to this, we have
	\[xz = P^1(wz) - P^1(w) z = P^1(wz) - w^p z.\]
However $w^p \in H^k$ is zero by the minimality of $x$ (see \cite[Lemma 2.3]{Kennard13}), and, applying Poincar\'e duality and periodicity once more, we have that $wz = 0$ since $wz \neq 0$ would imply a relation
	\[0 \neq xz = w'(wz) = (w'w)z.\]
But $w'w \in H^k$ is zero by the minimality of $x$ (again see \cite[Lemma 2.3]{Kennard13}), so $P^1(z) = 0$ as claimed.
\end{proof}

We are ready to conclude the proof from the fact that $H^{k-2} = 0$. Since $\beta(x) \in H^{k+1} \cong H^1 \cong 0$ and $P^1(x) \in H^{k+2(p-1)} \cong H^{k - 2} \cong 0$, \thm{thm:AdamsFactorization} implies
	\[x^p = P^p(x) = \be(w_0) + P^1(w_1)\]
for some $w_0, w_1\in H^*$. Since $k = 2p$, one has that $w_0 = x^{p-1} y_0$ and $w_1 = x^{p-1} y_1$ for some $y_0,y_1 \in H^*$ of positive degree. Using again the fact that $\beta(x)=0$ and $P^1(x) = 0$, we have by the Cartan formula that
	\[x^p = x^{p-1}\be(y_0) + x^{p-1} P^1(y_1)\]
and hence, by periodicity, that
	\[x = \be(y_0) + P^1(y_1).\]
Using periodicity once more, we have that the nonzero element $x$ generates $H^k$, which implies that $\be(y_0)$ or $P^1(y_1)$ is a nontrivial multiple of $x$. This contradicts the minimality of $x$ established at the beginning of the proof.

\bigskip\section{Free actions on positively curved rational spheres}\label{sec:DavisWeinbergerBurnside}\bigskip

In this section, we state the Davis--Weinberger factorization, which obstructs certain free actions by finite groups on rational homology spheres of dimension $4k+1$. We also discuss examples and properties of finite groups with periodic cohomology, including classifications due to Burnside, Wolf, and Madsen--Thomas--Wall. We also provide in Example \ref{exa:SCC} an important class of finite groups that act freely on spheres in a way that commutes with large isometric torus actions.

\begin{proposition}[Davis--Weinberger factorization]\label{pro:DavisWeinberger}
Let $M^{4k+1}$ be a closed, simply connected manifold with the rational homology of a sphere. Assume $M$ admits a Riemannian metric with positive sectional curvature. If $\pi$ is a finite group acting freely by isometries on $M$, then
	$\pi \cong \Z_{2^e} \times \Gamma$
for some $e \geq 0$ and some group $\Gamma$ with odd order.
\end{proposition}

\begin{proof}
The conclusion is the same as that of Theorem D in \cite{Davis83}, so it suffices to check the assumptions. Indeed, a theorem of Weinstein shows that the $\pi$--action preserves orientation (see \cite{Weinstein68}), hence $\pi$ acts trivially on the rational cohomology of $M$. In addition, the rational semicharacteristic $\sum_{i=0}^{2k} \dim H_i(M;\Q) = 1$,
so the proof is complete.
\end{proof}

The author would like to thank D. Cooper and D.D. Long for directing him to the work of Davis and Weinberger. The assumption that the dimension is of the form $4k+1$ is crucial. In fact, Cooper and Long proved that any finite group can act freely on a rational homology sphere of dimension $4k+3$, for any $k\geq 0$ (see \cite{CooperLong00}, cf. Browder--Hsiang \cite{BrowderHsiang78}).

This factorization immediately implies the following.

\begin{corollary}[Milnor condition]
Let $M^{4k+1}$ be a simply connected, positively curved rational homology sphere. Suppose $\pi$ is a finite group that acts freely by isometries on $M$. If a nontrivial involution in $\pi$ exists, then it is both unique and central.
\end{corollary}

Milnor proved that, if a finite group acts freely on a closed manifold with the mod $2$ homology of a sphere, then every involution is central (see Milnor \cite{Milnor57} or Bredon \cite[Chapter III.8.3]{Bredon72}). Combined with Smith's theorem (Theorem \ref{thm:Smith44}), it follows that there is at most one nontrivial involution. (For a different proof of Milnor's result, see Davis \cite[Corollary 7.5]{Davis83}).

\smallskip

Next, we discuss another important property of the Davis--Weinberger factorization, which we will use repeatedly. Suppose $\pi \cong \Z_{2^e} \times \Gamma$ for some $e\geq 0$ and some odd-order group $\Gamma$. Both of the following hold:
	\begin{enumerate}
	\item $\pi$ contains a copy of $\Z_p\times\Z_p$ for some prime $p$ if and only if $\Gamma$ does.
	\item $\pi$ contains a normal, cyclic subgroup of odd index $d$ if and only if $\Gamma$ does.
	\end{enumerate}
In the context of our proofs, $\pi$ is a finite group acting freely on a simply connected manifold. Once it is established that $\pi$ factors as $\Z_{2^e} \times \Gamma$ as above, the proofs proceed by restricting attention to odd-order groups.  Moreover, for such groups, we have the following classification of Burnside (see Burnside \cite{Burnside05} or Wolf \cite[Theorems 5.3.2 and 5.4.1]{Wolf}):

\begin{theorem}[Burnside classification]\label{thm:Burnside}
For a finite group $\Gamma$ of odd order, $\Gamma\not\supseteq \Z_p\times\Z_p$ for all primes $p$ if and only if
	\[\Gamma \cong \langle\al,\be \st \al^a=\be^b=1,\be\al\be^{-1}=\al^c\rangle\]
for some $a,b,c\geq 1$ such that $\gcd(a,b) = 1$, $\gcd(a,c-1) = 1$, and $c^b \equiv 1 \bmod{a}$.
\end{theorem}

\begin{definition}\label{def:Cd}
Call a triple $(a,b,c)$ of positive integers \textit{admissible} if $\gcd(a,b) = 1$, $\gcd(a,c-1)=1$, and $c^b \equiv 1 \bmod{a}$. For each admissible triple $(a,b,c)$, let $\Gamma(a,b,c)$ denote the group generated by $\al$ and $\be$ subject to the relations $\al^a=1$, $\be^b = 1$, and $\be\al\be^{-1} = \al^c$. For each $d \geq 1$, let $\mathcal C_d$ denote the collection of all groups $\Gamma(a,b,c)$ such that $d$ is the order of $c$ in $\Z_a^\times$.
\end{definition}

Note that the definition of $\mathcal C_d$ implies further that $d = 1$ if and only if $\Gamma$ is cyclic. For $d > 1$, we make the following observations. It follows from the relation $\be\al\be^{-1} = \al^c$ that $\al$ and $\be^d$ commute. Moreover, since $\gcd(a,b) = 1$, these elements have coprime order and hence generate a cyclic subgroup of index $d$. It is easy to see that this subgroup is normal in $\Gamma$. Altogether, we have that every $\Gamma \in \mathcal C_d$ contains a normal, cyclic subgroup of index $d$. By the observation above, $\pi \cong \Z_{2^e} \times \Gamma$ also contains a normal, cyclic subgroup of index $d$. 

\smallskip

In the classification of spherical space forms in dimension $n \equiv 1 \bmod{4}$, the fundamental groups that arise can be described in terms of the groups $\Gamma(a,b,c) \in \mathcal C_d$.

\begin{theorem}[Vincent]\label{thm:Wolf}
A finite group $\pi$ acts freely and linearly on $\s^{4k+1}$ only if $\pi \cong \Z_{2^e} \times \Gamma$ for some $e \geq 0$ and some group $\Gamma=\Gamma(a,b,c) \in \mathcal C_d$ such that $a$, $b$, and $d$ are odd, and such that every prime divisor of $d$ also divides $b/d$.

Conversely, if $\pi \cong \Z_{2^e} \times \Gamma$ for some odd-order group $\Gamma = \Gamma(a,b,c) \in \mathcal C_d$ where $d$ is an odd number all of whose divisors divide $b/d$, then $\pi$ acts freely and linearly on $\s^{2vd-1}$ for all odd $v \geq 1$.
\end{theorem}

\begin{proof}
The first statement follows from the Davis--Weinberger factorization together with Vincent's classification of spherical space form groups in dimensions $4k+1$ (see Wolf \cite[Theorems 5.4.1 and 5.6.1]{Wolf}). Alternatively, we can provide an argument based more directly on the classification of spherical space forms, and we do this here. Applying Vincent's theorem right away, we conclude that $\pi \cong \Gamma(a,b',c) \in \mathcal C_d$ for some admissible triple $(a,b',c)$ where $b'$ is not necessarily odd. The representations of $\Gamma(a,b',c)$ inducing free actions on $\s^{4k+1}$ have degree divisible by $2d$ (see \cite[Theorem 5.5.10]{Wolf}). Consequently, $2d$ divides $4k+2$ and hence $d$ is odd. Write $b' = 2^e b$ where $b$ is odd, and note that $d$ divides $b$. From the relations defining $\Gamma(a,b',c)$, it follows that $\alpha$ commutes with $\beta^d$ and hence with $\beta^{b}$. It follows that the subgroup generated by $\beta^{b}$ is isomorphic to $\Z_{2^e}$ and central in $\pi$, and moreover that $\pi$ factors as a direct product $\Z_{2^e} \times \Gamma\of{a,b,c^{2^e}}$, as claimed. Since $d$ is odd, the property that every prime divisor of $d$ also divides $b'/d$ carries over from $b'$ to $b$, so the proof of the first statement is complete.

For the converse, fix such a group $\pi \cong \Z_{2^e} \times \Gamma(a,b,c)$ where $\Gamma(a,b,c) \in \mathcal C_d$ and $d$ is odd. Note that both $c$ and $c^{2^e}$ generate the same subgroup in $\Z_a^\times$, so we may write $c$ as $\tilde c^{2^e}$ for some $\tilde c \in \Z_a^\times$. By the same isomorphism from the previous paragraph, we have that
	\[ \pi \cong \Z_{2^e} \times \Gamma\of{a,b,\tilde c^{2^e}} \cong \Gamma\of{a,2^eb,\tilde c}.\]
Using Vincent's classification again (specifically \cite[Theorem 5.5.10]{Wolf}), we conclude that $\pi$ acts freely and linearly on $\s^{2d-1}$.
\end{proof}

Regarding the classification of free, smooth, but not necessarily linear, actions by finite groups on the standard sphere, we have the following realization theorem due to Madsen, Thomas, and Wall (see \cite[Theorem 0.5]{MadsenThomasWall76} and \cite{MadsenThomasWall83}).

\begin{theorem}[Madsen--Thomas--Wall]\label{thm:MadsenThomasWall}
If $\pi$ is a finite group such that every abelian subgroup is cyclic and every involution is central, then $\pi$ acts freely and smoothly on a standard sphere. If, moreover, $\pi \cong \Z_{2^e} \times \Gamma$ for some odd-order $\Gamma \in \mathcal{C}_d$ with $d \geq 1$, then $\pi$ acts freely and smoothly on $\s^{2vd-1}$ for all $v \geq 1$.
\end{theorem}

The first claim is taken directly from \cite[Theorem 0.5]{MadsenThomasWall76}. For the second claim, we cite \cite{MadsenThomasWall83}, which shows that $\pi$ acts freely and smoothly on a standard sphere of dimension $2e(\pi) - 1$, where $e(\pi)$ is the Artin--Lam induction exponent. For our purposes, we only need to note that groups $\pi \cong \Z_{2^e} \times \Gamma$ with $\Gamma \in \mathcal C_d$ are $(2d)$--periodic (see Lemma \ref{lem:GroupCohomology}) and that $e(\pi) = d$ for such groups, which are known as Type I groups in the literature (see \cite{MadsenThomasWall83}).

\smallskip

We require one more collection of facts having to do with finite groups with periodic cohomology. The notion of periodic cohomology in the context of groups is the following: $H^*(\Gamma;\Z)$ is $(2d)$--periodic if $H^i(\Gamma;\Z) \cong H^{i+2d}(\Gamma;\Z)$ for all $i > 0$ (see, for example, Adem and Milgram \cite[Definition 6.1]{AdemMilgram04}). It is well known that a finite group having no subgroup of the form $\Z_p \times \Z_p$ has periodic cohomology, and so this class includes any group of the form $\pi \cong \Z_{2^e} \times \Gamma$ where $\Gamma \in \mathcal C_d$ for some $d \geq 1$.

\begin{lemma}[Results on finite groups with periodic cohomology]\label{lem:GroupCohomology}
Let $\Gamma$ be any finite group. 
	\begin{enumerate}
	\item If $\Gamma$ acts freely on $\s^n$, then $H^*(\Gamma;\Z)$ is $(n+1)$--periodic.
	\item If $\Gamma \in \mathcal C_d$, then $H^*(\Gamma;\Z)$ is $(2d)$--periodic.
	\item If $H^*(\Gamma;\Z)$ is periodic with periods $d_1$ and $d_2$, then it is $\gcd(d_1,d_2)$--periodic.
	\item If $H^*(\Gamma;\Z)$ is $2$--periodic, then $\Gamma$ is cyclic.
	\end{enumerate}
\end{lemma}

\begin{proof}
The first statement is Lemma 6.2 in Adem and Milgram \cite{AdemMilgram04}. The second is proved in Davis and Milgram \cite[p.~229]{DavisMilgram85}. The third statement is immediate from the definition. For the last statement, note that
	\[\Gamma/[\Gamma,\Gamma] \cong H_1(\Gamma;\Z) \cong H^2(\Gamma;\Z) \cong \Z/|\Gamma|,\]
where the last equality follows from \cite[Theorem 1.2]{DavisMilgram85}.
\end{proof}

Finally we come to the following class of examples showing that the conjecture in the introduction is sharp with respect to the symmetry assumption.

\begin{example}\label{exa:SCC}
Let $q$ be a prime dividing $\frac{n+1}{2}$. There exists a free, linear action by some $\Gamma \in \mathcal C_q$ on the round $n$--sphere that commutes with a linear torus action of rank $\frac{n+1}{2q}$.
\end{example}

\begin{proof}
Fix $q$ and $n$ such that $n + 1 \equiv 0 \bmod{2q}$. Choose positive integers $a$, $b$, and $c$ such that the following hold:
	\begin{itemize}
	\item $a$ is a prime congruent to one modulo $q$ (exists by Dirichlet's theorem),
	\item $b = q^2$, and
	\item $c$ is less than $a$ and has order $q$ in $\Z_a^\times$ (exists by Cauchy's theorem).
	\end{itemize}
It is straightforward to see that $(a,b,c)$ is an admissible triple and that $\Gamma=\Gamma(a,b,c) \in \mathcal C_q$. Moreover, it follows by Wolf's theorem (Theorem \ref{thm:Wolf}) that $\Gamma$ acts freely on $\s^{2q-1}$ and hence on $\s^n$ by taking an appropriate sum of representations.

It suffices to construct a free action on $\s^n$ that commutes with an action by a torus of dimension $\frac{n+1}{2q}$. To do this, consider the representation $\Gamma \to \Or(2q)$ given by
	\begin{eqnarray*}
	\al &\mapsto& \diag\of{R(1/a), R(c/a), R(c^2/a),\ldots,R(c^{q-1}/a)},\\
	\be &\mapsto& \of{\begin{array}{cc} 0 & \id \\ R(1/q) & 0\end{array}},
	\end{eqnarray*}
where $\al$ and $\be$ are as in Definition \ref{def:Cd} and $R(\theta)$ is the two-by-two rotation matrix with angle $2\pi\theta$. By taking $(m,n,r,d,n') = (a,b,c,q,q)$, it follows from \cite[Theorem 5.5.10]{Wolf} that this representation is fixed-point-free. Consequently, the sum of $\frac{n+1}{2q}$ of these representations induces a free action of $\Gamma$ on $\s^n$. Moreover, this action commutes with an action of $T^r$ on $\s^n$, where $r = \frac{n+1}{2q}$ and the action is induced by the representation $T^r \to O(n+1)$ given by
	\[(\theta_1,\ldots,\theta_r)
		\mapsto \diag\of{R(\theta_1),\ldots,R(\theta_1),\ldots,R(\theta_r),\ldots,R(\theta_r)}.\]

These examples show that the conjecture in the introduction is sharp with respect to the symmetry assumption. Moreover, note that the subaction by the diagonal $\sone \subseteq T^r$ is the (free) Hopf action.
\end{proof}

\bigskip\section{Free actions on positively curved rational spheres with circle symmetry}\label{sec:CircleObstruction}\bigskip

We now consider free actions on positively curved rational homology spheres with circle symmetry. The following obstruction is the main new tool we use in our analysis:

\begin{proposition}\label{pro:CircleObstruction}
Let $M^n$ be a closed, simply connected, odd-dimensional rational homology sphere with positive sectional curvature. If $\sone$ acts effectively by isometries, if $\Gamma$ is a group of odd order acting freely by isometries, and if these actions commute, then the following hold:
	\begin{enumerate}
	\item If $\Gamma \supseteq \Z_p \times \Z_p$ for some prime $p$, then $2p$ divides $\cod\of{M^{\sone}}$.
	\item If $\Gamma \in \mathcal C_d$ for some $d\geq 1$, then $2d$ divides $\cod\of{M^{\sone}}$. 
	\end{enumerate}
\end{proposition}

Here $\cod\of{M^{\sone}}$ denotes the codimension of the fixed-point set $M^{\sone}$. By convention, we set $\cod\of{M^{\sone}} = n + 1$ if the fixed-point set is empty. Also, $\mathcal C_d$ is as in Definition \ref{def:Cd}.

Note that this lemma applies to a closed, odd-dimensional, positively curved manifold whose universal cover is a rational homology sphere. Indeed, if $M$ admits an effective, isometric circle action, then its universal cover does as well. Moreover, this action commutes with the free $\pi_1(M)$--action. Combining this observation with Proposition \ref{pro:CircleObstruction} and the estimate $\cod\of{M^{\sone}} \leq n + 1$, we deduce the following consequence.

\begin{corollary}\label{cor:RongRationalSphere}
Let $M^n$ be a closed, positively curved Riemannian manifold whose universal cover is a rational homology sphere. If $M$ admits a non-trivial, isometric circle action, then $\pi_1(M) \supseteq \Z_p \times \Z_p$ for some prime $p$ only if $p \leq \frac{n+1}{2}$.
\end{corollary}

Restricting to dimension seven, Proposition \ref{pro:CircleObstruction} implies the following.

\begin{corollary}\label{cor:dim7}
Let $M^7$ be a closed, positively curved Riemannian manifold whose universal cover is a rational homology sphere. If a circle acts effectively by isometries on $M$, then every subgroup $\Gamma \subseteq \pi_1(M)$ of odd order is cyclic.
\end{corollary}

\begin{proof}
Denote the circle by $\sone$. The fixed-point set $M^{\sone}$ is another (connected) rational homology sphere of codimension $k \in \{2,4,6,8\}$ (see Smith \cite{Smith38} or Bredon \cite[III.10.2 and III.10.10]{Bredon72}). If $\Gamma \supseteq\Z_p\times\Z_p$ for some prime $p > 2$, then $2p$ divides $k$. This can only occur if $k = 6$, in which case $\Gamma$ acts on the circle $M^{\sone}$, a contradiction.

Since $\Gamma$ has odd order, Burnside's classification implies that $\Gamma \in \mathcal C_d$ for some odd $d\geq 1$. By Proposition \ref{pro:CircleObstruction} again, $d \leq 3$ with equality only if $\Gamma$ acts on the circle. Since groups in $\mathcal C_3$ are not cyclic, we conclude that $d = 1$ and hence that $\Gamma$ is cyclic.
\end{proof}

We now proceed to the proof of Proposition \ref{pro:CircleObstruction}. The proof is a modification of an argument in Sun and Wang (see \cite[Lemma 2.3]{SunWang09}, cf. Rong \cite{Rong05} and Rong--Wang \cite[Proposition 3.4]{RongWang05} and \cite{RongWang08}).

\begin{proof}[Proof of Proposition \ref{pro:CircleObstruction}]
Let $\bar M = M/\sone$, $N = M^{\sone}$, and $\bar N = N/\sone \subseteq \bar M$. As in the previous proof, since $M$ is a rational homology sphere, $N$ is another rational homology sphere of even codimension $k \leq n+1$, where $k = n+1$ corresponds to the case where $N$ is empty. In particular, $N$ is connected, so $\Gamma$ acts on $N$ since the actions of $\Gamma$ and $\sone$ commute.

Using cohomology with rational coefficients, the Smith--Gysin sequence (see Bredon \cite[III.10.5]{Bredon72})
	\[\cdots
		\longrightarrow H^i(M) 
		\longrightarrow H^{i-1}(\bar{M}, \bar N) \oplus H^i(\bar N)
		\longrightarrow H^{i+1}(\bar M,\bar N)
		\longrightarrow H^{i+1}(M)
		\longrightarrow \cdots\]
implies that $H^i(\bar M, \bar N; \Q)$ is $\Q$ for even indices $n-k+1 \leq i \leq n-1$ and $0$ for all other $i$. In particular, we have
	\[\chi(\bar M, \bar N) = \sum (-1)^i \dim H^i(\bar M, \bar N;\Q) = \frac{k}{2}.\]

We now consider the action of $\Gamma$ on $\bar M$, which is well defined since the actions of $\sone$ and $\Gamma$ commute. In both conclusions of the lemma, $\Gamma$ has a subgroup generated by some $\al$ and $\be$ subject to relations of the form $\al^a = 1$, $\be^b = 1$, and $\be\al\be^{-1} = \al^c$. In the first case $(a,b,c) = (p,p,1)$, and in the second case $(a,b,c)$ is any admissible triple (as in Definition \ref{def:Cd}) such that the order of $c \in \Z_a^\times$ is $d$.

Since $\al$ acts as a homeomorphism on $\bar M$, the induced map of $\al^2$ on $H^*(\bar M, \bar N;\Q)$ is the identity. It therefore follows by the Lefschetz fixed point formula that
	\[\chi(\bar M, \bar N) = \mathrm{Lef}(\al^2; \bar M, \bar N) = \chi(\bar M^{\al^2}, \bar N^{\al^2}).\]
Now $\al$ has odd order, so it and $\al^2$ generate the same subgroup. In particular, the fixed point sets of $\al$ and $\al^2$ coincide. Moreover, the fixed point set of the induced action on $\bar N$ is empty since $\al$ acts freely on $N$. Putting all of this together, we conclude
	\[\chi(\bar M^{\al}) = \frac{k}{2}.\]

We now bring $\be$ into the picture. Since $\be$ normalizes the subgroups $\langle \al \rangle$ and $\sone$, it acts on $\bar M^{\al}$. Of course, $\bar M^{\al}$ might have many components. We claim the following: If, for some $i$, $\be^i$ acts on some component of $\bar M^\al$, then $\be^i$ and $\al$ generate a cyclic subgroup. Given this claim for a moment, we conclude the lemma as follows:
	\begin{enumerate}
	\item In the first case of the lemma, $\be^i$ and $\al$ generate $\Z_p \times \Z_p$ for all $0 < i < p$, so $\be^i$ cannot act on any component of $\bar M^\al$ if $0 < i < p$. Therefore, $\be$ partitions the components of $\bar M^\al$ into groups of $p$ such that each group represents a unique homeomorphism type. In particular, $\chi(\bar M^\al) \equiv 0 \bmod{p}$, so $p$ divides $k/2$.
	
	\item In the second case, if $\be^i$ acts on some component of $\bar M^\al$, we would conclude from the claim that $\be^i$ and $\al$ commute. The relations defining $\Gamma(a,b,c)$ imply that this only occurs if $d$ divides $i$. As in the previous case, we partition $\bar M^\al$ into groups according to the orbits of the $\langle\be\rangle$-action on the set of components of $\bar M^\al$. In this case, we conclude that $\chi(\bar M^\al) \equiv 0 \bmod{d}$, so $d$ divides $k/2$.
	\end{enumerate}
It suffices now to prove the claim. Indeed, suppose that $\beta^i$ acts on some component $A \subseteq \bar{M}^\al$. Let $P:M \to \bar M$ be the projection. It follows exactly as in the proof of Lemma 2.3 in \cite{SunWang09} that, using a theorem of Rong (see \cite{Rong05}), $\beta^i$ acts on some circle orbit of $P^{-1}(A)$. But $\al$ acts on every circle orbit of $P^{-1}(A)$, hence the subgroup generated by $\be^i$ and $\al$ acts freely on a circle and is therefore cyclic.
\end{proof}

\bigskip\section{Free actions on positively curved rational  spheres with torus symmetry}\label{sec:Room}\bigskip

In this section, we apply the obstructions (Propositions \ref{pro:MorePeriodicity}, \ref{pro:DavisWeinberger}, and \ref{pro:CircleObstruction}) from the previous three sections to derive a further obstruction for positively curved rational homology spheres with torus symmetry. The following proposition essentially implies most of the theorems stated in the introduction in the special case where the universal cover of $M$ is a rational sphere. In particular, it together with Propositions \ref{pro:MorePeriodicity} and \ref{pro:DavisWeinberger} imply Theorem \ref{thm:RongWang5mod6}, and this proof is included in this section.

\begin{proposition}\label{pro:Room}
Let $M^n$ be a simply connected, rational sphere that admits a Riemannian metric with positive sectional curvature and $T^r$ symmetry. If $M$ admits a free action by a finite, odd-order group $\Gamma$ that commutes with the action by $T^r$, then the following hold:
	\begin{enumerate}
	\item If $\Gamma \supseteq \Z_p \times \Z_p$ for some prime $p$, then $r \leq \frac{n+1}{2p}$ and $r\leq \frac{n+1}{4p} + \frac{p}{2}$. 
	\item If $\Gamma \in \mathcal C_d$ for some $d \geq 1$, then $r \leq \frac{n+1}{2d}$.
	\end{enumerate}
Moreover, if the torus action has no fixed points, then $n+1$ is divisible by $2p$ in the first case and by $2d$ in the second.
\end{proposition}

The proof of Proposition \ref{pro:Room} uses Propositions \ref{pro:MorePeriodicity} and \ref{pro:CircleObstruction}. Before getting to it, we discuss two corollaries. The first improves the result (Theorem \ref{thm:FRWRW}) of Rong and Wang mentioned in the introduction.

\begin{corollary}[Theorem \ref{thm:RongWang5mod6}]\label{cor:RongWang5mod6}
If $M^n$ is a closed Riemannian manifold with positive sectional curvature and $T^r$ symmetry such that $n \geq 25$ and $r \geq \frac{n+1}{6}$, then one of the following occurs:
	\begin{itemize}
	\item $\pi_1(M)$ acts freely and linearly on $\s^3$.
	\item $r = \frac{n+1}{6}$, and $\pi_1(M)$ acts freely both by diffeomorphisms on a standard $\s^5$ and by isometries on a simply connected, positively curved mod $2$ and mod $3$ homology $5$--sphere. Moreover, the latter action commutes with an isometric circle action.
	\end{itemize}
\end{corollary}

As mentioned in Rong and Wang's paper, it is possible that the last conclusion could be improved by better understanding fundamental groups of positively curved $5$--manifolds with circle symmetry (see also Fang--Rong \cite{FangRong09}). The proof is a combination of Proposition \ref{pro:Room} and the results of Section \ref{sec:DavisWeinbergerBurnside}.

\begin{proof}[Proof of Corollary \ref{cor:RongWang5mod6} (Theorem \ref{thm:RongWang5mod6})]
We lift the positively curved metric and torus symmetry to $\tilde M$. The fundamental group acts freely by deck transformations on the universal cover, and this action commutes with the torus action on $\tilde M$.

Assume $\pi_1(M)$ is not a three-dimensional spherical space form. The proof of Theorem 1.1 in \cite{RongWang08} implies that $\pi_1(M)$ acts freely by isometries on a positively curved, simply connected, $5$-dimensional rational homology sphere, and that this action commutes with an isometric circle action. Moreover, this only arises in the case where $\tilde M$ contains a pair of transversely intersection, totally geodesic submanifolds of codimension six (see \cite[Remark 4.1]{RongWang08} regarding Case 3.(a) of the proof of Statement (2.3.1) of Lemma 2.3).

The first of these conclusions implies that $\pi_1(M)$ admits a Davis--Weinberger factorization $\pi_1(M) \cong \Z_{2^e} \times \Gamma$ for some $e \geq 0$ and some odd-order group $\Gamma$ (see Proposition \ref{pro:DavisWeinberger}). If $\Gamma$ is trivial, then $\pi_1(M)$ is cyclic and hence is a three-dimensional space form group. We may assume therefore that $\Gamma$ is not trivial. By Synge's theorem, $n$ is odd. 

Proposition \ref{pro:MorePeriodicity} now implies that $\tilde M$ is a mod $3$ homology sphere and hence that $\Gamma \not\supseteq \Z_3 \times \Z_3$. Note also that $\tilde M$ is mod $2$ homology sphere by earlier work of the author (specifically, \cite[Proposition 1.3]{Kennard13}). This completes the part of the claim regarding a free, isometric action of $\pi_1(M)$.

Next, we have that $\tilde M$ is a rational homology sphere, so Proposition \ref{pro:Room} applies. Since $r \geq \frac{n+1}{6}$, there is no subgroup $\Z_p \times \Z_p \subseteq \Gamma$ with $p > 3$. Combining this with the previous paragraph, we conclude that every abelian subgroup of $\Gamma$ is cyclic. By Burnside's classification (Theorem \ref{thm:Burnside}), $\Gamma \cong \Gamma(a,b,c) \in \mathcal C_d$ for some $d \geq 1$. But now Proposition \ref{pro:Room} applies again, giving the estimate $d \leq \frac{n+1}{2r} \leq 3$. Since we have assumed $\pi_1(M)$ is not cyclic, $d \neq 1$. Since $d$ is odd, we have $d = 3$ and $r = \frac{n+1}{6}$. It now follows from the realization theorem of Madsen, Thomas, and Wall (Theorem \ref{thm:MadsenThomasWall}) that $\pi_1(M)$ acts freely by diffeomorphisms on $\s^5$.
\end{proof}

The second result is an interpretation of Proposition \ref{pro:Room} in dimensions $9$ and $13$. Here, an easy argument actually produces a sharp conclusion. In fact, it establishes the conjecture from the introduction in the smallest possible dimension, $2q-1$, when $q \in \{5, 7\}$, under the additional assumption that the universal cover of the manifold is a rational sphere. The proof also foreshadows the proof of Proposition \ref{pro:Room} in higher dimensions.

\begin{corollary}\label{cor:dim9and13}
Let $M^n$ be a closed, positively curved Riemannian manifold whose universal cover is a rational sphere. If $n \in \{9,13\}$ and $M$ admits $T^2$ symmetry, then $\pi_1(M)$ is cyclic.
\end{corollary}

\begin{proof}
As in the proof of the previous corollary, we lift to the universal cover $\tilde M$. By Proposition \ref{pro:DavisWeinberger}, $\pi_1(M)$ factors as $\Z_{2^e} \times \Gamma$ for some odd-order group $\Gamma$.

First suppose that $\Gamma \supseteq \Z_p \times \Z_p$ for some prime $p$. Consider any sequence $\tilde M \supseteq \tilde M^{T^1} \supseteq \tilde M^{T^2}$ where both inclusions are strict. Each fixed-point set is another rational sphere of odd dimension, hence Proposition \ref{pro:CircleObstruction} applies to both inclusions. In particular, each inclusion has codimension at least $2p$. Since $\cod\of{\tilde M^{T^2}} \leq \dim(\tilde M) + 1 \leq 14$, this implies $2p + 2p \leq 14$, or $p = 3$. Moreover, if $p = 3$, then both codimensions equal six, $n = 13$, and hence $\tilde M^{T^2}$ is a circle. Since the action by $\Gamma$ restricts to $\tilde M^{T^2}$, this is a contradiction.

Suppose therefore that every abelian subgroup of $\Gamma$ is cyclic. By Burnside's classification (Theorem \ref{thm:Burnside}), $\Gamma \in \mathcal C_d$ for some (odd) $d\geq 1$. Applying the above argument once more implies $d = 1$ and hence that $\Gamma$ and $\pi_1(M)$ are cyclic.
\end{proof}

We spend the rest of this section on the proof of Proposition \ref{pro:Room}. First, we note that the last claim follows immediately from Proposition \ref{pro:CircleObstruction}. Indeed, when the torus has no fixed point, there is a circle inside whose fixed-point set is empty and hence has codimension $n + 1$, so Proposition \ref{pro:CircleObstruction} implies the result.

Second, by Berger's theorem (Theorem \ref{thm:Berger}), either $T^r$ or a codimension-one subtorus $T^{r-1} \subseteq T^r$ fixes some point $x \in M$. The subtorus acts on the normal sphere at that fixed point, and an inductive argument shows that there exists a sequence
	\[T^0 \subseteq T^1 \subseteq \cdots \subseteq T^{r-1}\]
such that the inclusion $M^{T^{i}} \subseteq M^{T^{i-1}}$ has positive codimension for all $1 \leq i \leq r - 1 $. Since $\cod\of{M^{T}} = n+1$ when $M^{T}$ is empty, we have that the inclusion $M^{T^r} \subseteq M^{T^{r-1}}$ has positive codimension as well.

Third, the fact that $M$ is a rational homology sphere implies that each $M^{T^i}$ is another (possibly empty) rational homology sphere of odd dimension. Indeed this follows by Smith's theorem as in the proof of Proposition \ref{pro:CircleObstruction}. In particular, each $M^{T^i}$ is connected, and the action by $\Gamma$ restricts to each $M^{T^i}$. Proposition \ref{pro:CircleObstruction} therefore implies that inclusion $M^{T^i} \supseteq M^{T^{i+1}}$ has codimension at least $2p$ (resp. $2d$) in Case 1 (resp. Case 2). The codimension of $M^{T^r}$ is therefore at least $2pr$ (resp. $2dr$). On the other hand, $\cod\of{M^{T^r}} \leq n-1$ if it $M^{T^r}$ is nonempty, and $\cod\of{M^{T^r}} = n + 1$ by convention otherwise. This concludes the proof that $2pr \leq n+1$ in Case 1 (resp. $2dr \leq n+1$ in Case 2).

It suffices to prove that $r \leq \frac{n+1}{4p} + \frac{p}{2}$ in Case 1. For this, we refine the above argument using Proposition \ref{pro:MorePeriodicity}. Set $k_0 = 4p$. Let $j \geq 0$ be maximal such that there exists a sequence
	\[T^0 \subseteq T^1 \subseteq \cdots \subseteq T^r\]
of subtori such that the following hold: $T^{r-1}$ fixes a point $x\in M$, the induced action of $T^r$ on submanifold $M_i = M^{T^i}$ has $i$-dimensional kernel for all $i$, and $\ave(k_0,k_1,\ldots,k_j) \geq 4p$, where $\ave$ denotes the average and $k_i = \cod\of{M_i\subseteq M_{i-1}}$.

Note that Proposition \ref{pro:CircleObstruction} implies that $k_i = 2p$ or $k_i \geq 4p$ for every $i$. Using this fact together with the maximality of $j$, we claim the following:
\begin{lemma}\label{lem:kj} The following hold:
	\begin{itemize}
	\item $k_0 + k_1 + \ldots + k_j = 4p(j+1)$.
	\item If $r - j \geq 3$, then there exists $T^j \subseteq H \subseteq T^{j+2}$ such that $M^{T^{j+1}}$ and $M^H$ have codimension $2p$ in $M^{T^j}$ and intersect transversely with intersection $M^{T^{j+2}}$.
	\end{itemize}
\end{lemma}

\begin{proof}
For the first statement, note that $\ave(k_0,\ldots,k_j) \geq 4p$ implies $k_0 + \ldots + k_j \geq 4p(j+1)$. On the other hand, if $k_0 + \ldots + k_j > 4p(j+1)$, then $k_0 + \ldots + k_j \geq 4p(j+1) + 2p$. In particular, $\ave(k_0,\ldots,k_{j+1}) \geq 4p$, a contradiction to the maximality of $j$.

For the second claim, suppose that $r - j \geq 3$. Observe that $k_{j+1} = 2p$, since $k_{j+1} \geq 4p$ would imply $\ave(k_0,\ldots,k_{j+1}) \geq 4p$, another contradiction to maximality. Similarly, $k_{j+2} \leq 4p$.

Next, since $j+2 \leq r - 1$, $T^{j+2}$ has a fixed point $x \in M$. Consider the isotropy representation of $T^{j+2}$ at $x$. There exists an $(j+1)$--dimensional torus $H$ between $T^j$ and $T^{j+2}$ such that $H$ fixes some direction tangent to $M^{T^j}$ and normal to $M^{T^{j+1}}$ at $x$. Note that $M^H$ lies strictly between $M^{T^{j+2}}$ and $M^{T^j}$. By Proposition \ref{pro:CircleObstruction}, both of these inclusions have codimension divisible by $2p$. The maximality of $j$ then implies $\cod\of{M^H \subseteq M_j} = 2p$. Finally, since $H$ and $T^{j+1}$ generate $T^{j+2}$, it follows that $M^{T^{j+2}}$ is the transverse intersection in $M^{T^j}$ of $M^H$ and $M^{T^{j+1}}$.
\end{proof}

We are ready to complete the proof. Let $n_i$ denote the dimension of $M^{T^{i}}$ for all $i$. Since induced $T^r$ action on $M_i$ has $i$-dimensional kernel, $M_i$ admits an isometric, effective $(r-i)$--dimensional torus action that commutes with the action of $\Gamma$.

First, suppose that $r - j \geq p + 1$. By the second part of Lemma \ref{lem:kj}, $M_j$ contains a pair of transversely intersecting, totally geodesic submanifolds of codimension $2p$. Moreover, since $\cod(M_i \supseteq M_{i+1}) \geq 2p$ for all $i \geq j$ and $\dim(M_r) \geq -1$ (with equality corresponding to the case where $M_r = M^{T^r}$ is empty), we have
	\[\dim(M_j) \geq (2p)(r-j) - 1 \geq 2p(p+1) - 1 > \max(19, 2p^2).\]
Proposition \ref{pro:MorePeriodicity} now implies that $M_j$ is a mod $p$ homology sphere and hence that $\Z_p \times \Z_p$ cannot act freely on $M_j$. This is a contradiction.

We may assume therefore that $r - j \leq p$. Here we can show that the codimension of $M_r$ is large since $\cod(M_i \supseteq M_{i+1}) = 4p$ for all $i < j$. Specifically, we have
	\[n+1 \geq \cod(M_r) \geq (4p)j + (2p)(r-j) \geq (4p)(r-p) + (2p)(p) = 4pr - 2p^2.\]
Solving for $r$ proves the required bound.

\bigskip\section{Proof of Theorem \ref{thm:CCSlog}}\label{sec:CCSlog}\bigskip

So far, we have mainly considered simply connected manifolds that admit free actions by finite groups. Clearly the application is to the action of the fundamental group on the universal cover. The main results were the obstructions presented in Propositions \ref{pro:MorePeriodicity}, \ref{pro:CircleObstruction}, and \ref{pro:Room}. From now on, $M$ will denote a (not necessarily simply connected) Riemannian manifold, and $\tilde M$ will denote the universal cover. The purpose of this section is to prove Theorem \ref{thm:CCSlog} from the introduction. The theorem is included in the following:

\begin{theorem}[Theorem \ref{thm:CCSlog}]\label{thm:CCSlogPLUS}
Let $M^n$ be a closed Riemannian manifold with positive sectional curvature and $T^r$ symmetry. If $n \equiv 1 \bmod{4}$ and $r > \log_{4/3}\pfrac{n+3}{6}$, then $\pi_1(M) \cong \Z_{2^e} \times \Gamma$ for some $e\geq 0$ and some odd-order group $\Gamma$. Moreover, the following hold:
	\begin{enumerate}
	\item If $\Gamma \supseteq \Z_p \times \Z_p$ for some prime $p$, then $r \leq \frac{n+1}{2p}$ and $r \leq \frac{n+1}{4p} + \frac{p}{2}$.
	\item If $\Gamma \not\supseteq \Z_p \times \Z_p$ for all primes $p$, then $\Gamma \in \mathcal C_d$ for some positive, odd integer $d \leq \frac{n+1}{2r}$. In particular, $\pi_1(M)$ contains a normal, cyclic subgroup of index $d$ and acts freely and smoothly on $\s^{2d-1}$.
	\end{enumerate}
\end{theorem}

First, we make a few remarks. Using Burnside's classifiation (Theorem \ref{thm:Burnside}), the property that $\Gamma \not\supseteq\Z_p \times \Z_p$ for all primes $p$ implies that $\Gamma \in \mathcal C_d$ for some $d\geq 1$. Since $\Gamma$ has odd order, $d$ is also odd by the definition of $\mathcal C_d$ (see Definition \ref{def:Cd}). Moreover, the properties of the class $\mathcal C_d$ imply the statement about the normal, cyclic subgroup of index $d$ (see Lemma \ref{lem:GroupCohomology}), and the Madsen--Thomas--Wall classification implies the existence of a free action by diffeomorphisms on $\s^{2d-1}$ (see Theorem \ref{thm:MadsenThomasWall}). With these comments in mind, we note that Conclusion (2) is equivalent to the following:
	\begin{enumerate}
	\item[(2')] If $\Gamma \in \mathcal C_d$ for some $d \geq 1$, then $d \leq \frac{n+1}{2r}$.
	\end{enumerate}

We proceed to the proof of Theorem \ref{thm:CCSlogPLUS}. The proof is a fairly straightforward combination of Proposition \ref{pro:Room} (for the case where $\tilde M$ is a rational sphere) and an inductive argument using Proposition \ref{pro:AKlog} (for the case where $\tilde M$ is not a rational sphere). However the argument only starts to work easily in large dimensions, so for simplicity we take a moment prove the following lemma. Note that $r > \log_{4/3}\pfrac{n+3}{6}$ in dimensions up to $97$ implies that $r \geq \sqrt{n-1}$.

\begin{lemma}[Dimensions up to $97$]\label{lem:CCSsqrt}
If $M^n$ is as in Theorem \ref{thm:CCSlogPLUS} except that we assume that $n \leq 97$ and $r \geq \sqrt{n-1}$, then $\pi_1(M) \cong \Z_{2^e} \times \Gamma$ for some $e \geq 0$ and some odd-order group $\Gamma$. Moreover, $\Gamma \not\supseteq \Z_p \times \Z_p$ for all primes $p$, and $\Gamma \in \mathcal C_d$ for some odd $d \leq \min\of{3, \frac{n+1}{2r}}$. In particular, $\pi_1(M)$ acts freely and smoothly on $\s^5$. In fact, $\pi_1(M)$ is cyclic if $n \leq 33$.
\end{lemma}

The last statement will not be required for the rest of the proof, but it follows easily from the proof below. It provides a new result in dimension $33$, as far as the author can tell.

\begin{proof}[Proof of Lemma \ref{lem:CCSsqrt}]
We lift the metric and torus action to the universal cover $\tilde M$, and we consider the free, isometric action of $\pi_1(M)$ on $\tilde M$ that commutes with the torus action.

First suppose that $\tilde M$ is a rational homology sphere. By Proposition \ref{pro:DavisWeinberger}, $\pi_1(M)$ admits a Davis--Weinberger factorization as $\Z_{2^e} \times \Gamma$ for some $e \geq 0$ and some odd-order group $\Gamma$. Consider the subaction by $\Gamma$ on $\tilde M$. Proposition \ref{pro:Room} implies that $\Gamma \supseteq \Z_p \times \Z_p$ for some prime $p$ only if $r \leq \frac{n+1}{2p}$ and $r \leq \frac{n+1}{4p} + \frac{p}{2}$.  Since $n \leq 97$ and $r \geq \sqrt{n-1}$, the first of these inequalities implies $p = 3$ and $n \geq 37$. Given that $p = 3$, the second inequality implies a contradiction. Hence $\Gamma \not\supseteq \Z_p \times \Z_p$ for any prime $p$. As argued above before the statement of Conclusion (2'), we have that $\Gamma \in \mathcal C_d$ for some odd $d \geq 1$ and that $\pi_1(M)$ has a normal, cyclic subgroup of index $d$ and acts freely by diffeomorphisms on $\s^{2d-1}$. Applying Proposition \ref{pro:Room} again, we have that $d \leq \frac{n+1}{2r}$. As with the estimate on $p$ above, we have $d \leq 3$, with equality only if $n \geq 37$. This completes the proof in this case.

Now suppose that $\tilde M$ is not a rational homology sphere. For $n \leq 21$ or $n = 29$, the symmetry assumption implies $r \geq \frac{n+1}{6} + 1$, so $\pi_1(M)$ is cyclic (and hence $\pi_1(M) \cong \Z_{2^e} \times \Gamma$ for some cyclic group $\Gamma \in \mathcal C_1$) by the result of Frank--Rong--Wang (Theorem \ref{thm:WFRW}). Moreover, for $n = 25$, the symmetry assumption implies $r \geq 5$, so $\pi_1(M)$ is cyclic by a theorem of Wang mentioned in the introduction (see \cite[Theorem B]{Wang07}). Assume therefore that $33 \leq n \leq 97$. In this range, an analogue of Proposition \ref{pro:AKlog} holds under the assumption that $s = r - 2 \geq \sqrt{n-1} - 2$. Indeed, as in the proof of Proposition \ref{pro:AKlog}, it suffices to check Inequality \ref{ine:1.1}. This can be checked on a case-by-case basis. Given this analogue, our assumption that $\tilde M$ is not a rational sphere implies the existence of an involution $\iota \in T^r$ and a component $N \subseteq \tilde M^\iota$ such that $\cod(N) \equiv 0 \bmod{4}$, $\frac{n-1}{4} < \cod(N) \leq \frac{n-1}{2}$, and $\dim\ker(T^s|_N) \leq 1$. Since $\tilde M$ is a not a rational sphere, we may assume further that $\dim\ker(T^r|_N) \leq 1$ by Lemma \ref{lem:dk2}. In particular, $\cod(N)$ is divisible by four and at least $\frac{n+3}{4}$, and the rank of torus symmetry of $N$ is at least $r - 1 \geq \sqrt{n-1} - 1$. Combining these estimates, it can be checked on a case-by-case basis that $N$ satisfies the symmetry assumption of the lemma. Since $\cod(N) \leq \frac{n-1}{2}$, $\pi_1(M)$ acts on $N$ and so, by induction, it follows that $\pi_1(M) \cong \Z_{2^e} \times \Gamma$ for some odd-order $\Gamma \in \mathcal C_d$ such that $d \leq \min\of{3, \frac{\dim(N) + 1}{2(r-1)}}$. Since $n \geq 17$ and $r \geq \sqrt{n-1}$, this implies $d \leq \min\of{3, \frac{n+1}{2r}}$. Moreover, $d = 3$ cannot occur if $n = 33$, so the proof is complete.
\end{proof}

With this lemma out of the way, we have that Theorem \ref{thm:CCSlogPLUS} holds for dimensions up to $97$. It therefore suffices to prove it for dimensions $n \geq 101$.

\begin{proof}[Proof of Theorem \ref{thm:CCSlogPLUS} for $n \geq 101$]
We again lift to the universal cover $\tilde M$. We also recall the discussion above that Conclusion (2) and Conclusion (2') are equivalent. With this in mind, the theorem follows immediately when $\tilde M$ is a rational homology sphere by the Davis-Weinberger factorization (Proposition \ref{pro:DavisWeinberger}) and Proposition \ref{pro:Room}. We assume therefore that $\tilde M$ is not a rational homology sphere.

By Berger's lemma, some $T^{r-1} \subseteq T^r$ has a fixed point $x \in \tilde M$. Consider the map $\phi:\Z_2^{r-1} \to \Z_2^{(n-1)/2} \subseteq \SO(T_x \tilde M)$ induced by the isotropy representation at $x$. For $\iota \in \Z_2^{r-1}$, the Hamming weight of $\phi(\iota) \in \Z_2^{(n-1)/2}$ equals half the codimension of the fixed-point component $N \subseteq \tilde M^\iota$ containing $x$. Hence an even Hamming weight corresponds to a fixed-point component of codimension divisible by four. Since the map $\phi$ is linear, there exists $\Z_2^{r-2} \subseteq \Z_2^{r-1}$ such that every $\iota \in \Z_2^{r-2}$ satisfies $\cod(N) \equiv 0 \bmod{4}$ where $N \subseteq \tilde M^\iota$ is the component containing $x$. Set $s = r - 2$. 

Since $\tilde M$ is not a rational sphere, Proposition \ref{pro:AKlog} implies the existence of an $\iota \in \Z_2^s$ such that some component $N$ of its fixed-point set satisfies $\cod(N) \equiv 0 \bmod{4}$, $\cod(N) \geq \frac{n+3}{4}$, $\cod(N) \leq \frac{n-1}{2}$, and $\dim\ker(T^s|_N) \leq 1$. By Lemma \ref{lem:dk2}, we may assume further that $\dim\ker(T^r|_N) \leq 1$ since otherwise $\tilde M$ would be a rational sphere. The actions of $\pi_1(M)$ and $T^s$ commute, so $\pi_1(M)$ acts on the fixed-point set $\tilde M^\iota$. Moreover, Frankel's theorem and the property that $\cod(N) \leq \frac{n-1}{2}$ implies that $N \subseteq \tilde M^\iota$ is the only component of dimension greater than $\frac{n}{2}$, so $\pi_1(M)$ acts on $N$. Moreover, the other properties enjoyed by $N$ imply that $N$ satisfies the induction hypothesis. In particular, the estimates $\cod(N) \geq \frac{n+3}{4}$ and $\dim\ker(T^r|_N) \leq 1$ imply that $N$ admits $T^s$ symmetry with $s \geq r - 1 > \log_{4/3}\pfrac{\dim(N) + 3}{4}$. By induction, we conclude that $\pi_1(M) \cong \Z_{2^e} \times \Gamma$ for some odd-order group $\Gamma$.

We proceed to the proof of Conclusion (1). Suppose that $\Gamma \supseteq \Z_p \times \Z_p$ for some prime $p$. By induction again, Conclusion (1) holds for $N$. This implies both of the following inequalities:
	\begin{eqnarray}
	r - 1 &\leq& \frac{\dim(N) + 1}{2p},\\\label{eqn:IH2p}
	r - 1 &\leq& \frac{\dim(N) + 1}{4p}+\frac{p}{2}.\label{eqn:IH4p}
	\end{eqnarray}

Note that the bound on $r$ implies $r \geq 4$. Hence $\frac{3}{4}r \leq r - 1$. In addition, the right-hand side of Inequality \ref{eqn:IH2p} is at most $\frac{3}{4}\pfrac{\dim(N) + 1}{2p}$ since $\cod(N) > \frac{n+1}{4}$. Putting this together, we conclude $r \leq \frac{n+1}{2p}$, which is the first part of Conclusion (1).

It now suffices to prove $r \leq \frac{n+1}{4p} + \frac{p}{2}$. We break the proof into cases.
	\begin{itemize}
	\item If $n + 1 \leq 4p(p-2)$, then the estimate 
		\[\dim(N) + 1 < \frac{3}{4}(n+1) \leq \frac{n+1}{2} + {p(p-2)}\]
	together with Inequality \ref{eqn:IH2p} imply
		\[r \leq \frac{\dim(N) + 1}{2p} + 1 
			< \frac{n+1}{4p}  + \frac{p}{2}.\]
	\item If $n + 3 \geq 16p$, then $\cod(N) \geq \frac{n+3}{4} \geq 4p$ and hence Inequality \ref{eqn:IH4p} implies
		\[r \leq \frac{\dim(N) + 1}{4p} + \frac{p}{2} + 1 \leq \frac{n+1}{4p} + \frac{p}{2}.\]
		\item If $4p(p-2) < n+ 1$ and $n+3 < 16p$, then $p \leq 5$ and hence $n < 16p-3 \leq 77$, a contradiction. Hence one of the previous two cases occurs.
	\end{itemize}

It suffices to prove Conclusion (2). By the remarks following Theorem \ref{thm:CCSlogPLUS}, this is equivalent to proving Conclusion (2'). Suppose then that $\Gamma \in \mathcal C_d$ for some $d \geq 1$. The task is to show that $r \leq \frac{n+1}{2d}$. But this follows by induction using the argument for Conclusion (1) that showed $r \leq \frac{n+1}{2p}$, so the proof is complete.
\end{proof}

\bigskip\section{Proof of Theorem \ref{thm:CCS12}}\label{sec:CCS}\bigskip

Theorem \ref{thm:CCS12} has two parts. The first part involves the assumptions that $r > \frac{n+1}{8} + 1$ and $n \geq 23$, and the result is a direct generalization of Theorems A and B in Sun--Wang \cite{SunWang09}. We discuss the proof of this first. For convenience, we restate our theorem here:

\begin{theorem}\label{thm:SunWangPLUS}
Let $M^n$ be a closed, positively curved manifold with $T^r$ symmetry. If $n \geq 23$ and $r \geq \frac{n+3}{8} + 1$, then $\pi_1(M)$ acts freely and smoothly on a standard sphere.
\end{theorem}

Besides the application of the realization theorem of Madsen, Thomas, and Wall (Theorem \ref{thm:MadsenThomasWall}), the only substantive improvement here upon the results of Sun and Wang is the conclusion that $\pi_1(M) \not\supseteq \Z_3 \times \Z_3$ under the conditions of the theorem. The place in Sun and Wang's argument where $\Z_3 \times\Z_3$ cannot be excluded is in Lemma 1.5 of their paper. Our strategy for proving Theorem \ref{thm:SunWangPLUS} is not to improve Lemma 1.5 in their paper. Instead, we look in their proof where Lemma 1.5 is applied, and we provide an alternative argument based on Proposition \ref{pro:MorePeriodicity}.

\begin{proof}
By Theorem \ref{thm:MadsenThomasWall}, it suffices to show that $\pi_1(M)\not\supseteq\Z_p \times \Z_p$ for every prime $p$ and that every involution $\pi_1(M)$ is central. The second of these claims follows by Theorem B in \cite{SunWang09}. It suffices to show that $\pi_1(M) \not\supseteq \Z_p\times\Z_p$ for any prime $p$.

We follow the proof in \cite[Section 1]{SunWang09} and adopt the authors' notation for the remainder of this proof. In particular, let $F \subseteq \tilde M$ denote the fixed-point component of an involution, chosen so that the codimension $\cod(F)$ is minimal. If $\cod(F)$ is two or four, the proof carries through without change. In particular, $\pi_1(M)$ is cyclic in the first case and isomorphic to a three-dimensional spherical space form group in the second. Next, if $\cod(F) \geq 8$ or if $\cod(F) = 6$ and $F$ is not fixed by a circle in $T^r$, the proof again carries through using either the induction hypothesis if $\dim(F) \geq 23$ or previous lower-dimensional results. 

This leaves the case where $\cod(F) = 6$ and $F$ is fixed by a circle in $T^r$. In this case, the authors apply \cite[Lemma 1.5]{SunWang09} to conclude that $\pi_1(M) \not\supseteq \Z_p \times \Z_p$ for any prime $p \neq 3$. To exclude the possibility that $\pi_1(M) \supseteq \Z_3 \times \Z_3$, we apply the following argument. By \cite[Lemma 1.1]{SunWang09}, there is a second fixed-point component $F'$ of an involution in $T^r$ with $\cod(F) < \frac{n}{2}$ that is fixed by at most a one-dimensional subtorus of $T^r$. The proof is again complete if $\cod(F') \neq 6$ by repeating the above arguments, so we assume $\cod(F') = 6$.

If $F \cap F'$ is not transverse, then the intersection is a codimension-two, totally geodesic submanifold of either $F'$ or of the fixed-point component of the product of the two involutions under consideration. It follows immediately that $\pi_1(M)$ is cyclic in this case (e.g., see \cite[Lemma 1.4]{SunWang09}). If, on the other hand, the intersection $F \cap F'$ is transverse, then we are in the setting of Proposition \ref{pro:MorePeriodicity}, which implies that $\tilde M$ is a mod $3$ homology sphere and hence that $\pi_1(M) \not\supseteq \Z_3 \times \Z_3$, as claimed.
\end{proof}

We proceed to prove the second part of Theorem \ref{thm:CCS12}. It is an immediate consequence of the following result (take $q = 3$).

\begin{theorem}\label{thm:CCS4q}
Let $M^n$ be a closed, positively curved Riemannian manifold with $T^r$ symmetry. If $n \equiv 1 \bmod{4}$ and $r \geq \frac{n-1}{4q} + q$ for some real number $q > 0$, then $\pi_1(M) \cong \Z_{2^e} \times \Gamma$ where $\Gamma$ has odd order. In particular, every involution is central. Moreover, the following hold:
	\begin{enumerate}
	\item If $\Gamma \supseteq \Z_p \times \Z_p$ for some prime $p$, then $2 < p < q$.
	\item If $\Gamma \not\supseteq \Z_p \times \Z_p$ for any prime $p$, then $\Gamma \in \mathcal C_d$ for some odd $d\geq 1$ such that $d< 2q$. As a consequence, $\pi_1(M)$ contains a normal, cyclic subgroup of index $d$ and acts freely and smoothly on the standard sphere of dimension $2d-1$.
	\end{enumerate}
\end{theorem}

This result provides an obstruction to the existence of subgroups $\Z_p \times \Z_p \subseteq \pi_1(M)$ for all primes $p$ larger than a prescribed number $q$, and the symmetry assumption is precisely what is required to do this, given the techniques of this paper.

The proof of this theorem is a straightforward application of Theorem \ref{thm:CCSlogPLUS} and Lemma \ref{lem:CCSsqrt}. 

\begin{proof}
First observe that $n \geq 5$ and $0 < q \leq 1$ implies that $r \geq \frac{n}{4} + 1$, hence $\pi_1(M)$ is cyclic by Wilking's theorem (Theorem \ref{thm:WFRW}). We assume from now on that $n \geq 5$ and $q \geq 1$. In particular, the symmetry assumption implies that $r > \frac{n+1}{4q}$.

If $n \leq 97$, then the result follows immediately from Lemma \ref{lem:CCSsqrt} since
	\[r \geq \frac{n-1}{4q} + q \geq \sqrt{n-1}.\]

Suppose now that $n \geq 101$. In this range, $r \geq \sqrt{n-1}$ implies that $r \geq \log_{4/3}\pfrac{n+3}{6}$, so Theorem \ref{thm:CCSlogPLUS} applies. The  conclusion of Theorem \ref{thm:CCSlogPLUS} is that $\pi_1(M)$ factors as $\Z_{2^e} \times \Gamma$ for some odd-order group $\Gamma$ such that
	\begin{enumerate}
	\item if $\Gamma \supseteq \Z_p\times\Z_p$ for some prime $p$, then $r \leq \frac{n+1}{2p}$ and $r \leq \frac{n+1}{4p} + \frac{p}{2}$, and
	\item if $\Gamma \not\supseteq \Z_p\times \Z_p$ for any prime $p$, then $\Gamma \in \mathcal C_d$ for some odd $d \geq 1$ such that $d\leq \frac{n+1}{2r}$. In particular, $\pi_1(M)$ contains a normal, cyclic subgroup of index $d$ and acts freely and smoothly on $\s^{2d-1}$.
	\end{enumerate}
Suppose for a moment that $\Gamma \supseteq \Z_p \times \Z_p$ for some prime $p$. By (1), we conclude that $r \leq \frac{n+1}{2p}$ and $r \leq \frac{n+1}{4p}+\frac{p}{2}$. For $p \geq 2q$, this implies that 
	\[r \leq \frac{n+1}{2p}                       \leq \frac{n+1}{4q},\]
which contradicts the bound on $r$ and the assumption that $q \geq 1$. Similarly, $q \leq p \leq 2q - 1$ implies that
	\[r \leq \frac{n+1}{4p} + \frac{p}{2}   \leq   \frac{n+1}{4q} + q - \frac{1}{2},\]
which leads to a similar contradiction. Hence $p$ must be less than $q$, as claimed in Conclusion (1) of the theorem. The proof of Conclusion (2) is immediate from the combination of the estimates $d \leq \frac{n+1}{2r}$ and $r > \frac{n+1}{4q}$ above, so the proof is complete.
\end{proof}

\bigskip\section{Proof of Theorem \ref{thm:SCC}}\label{sec:SCC}\bigskip

We reproduce the first part of Theorem \ref{thm:SCC} here for easy reference.

\begin{theorem}\label{thm:SCCintegral}
Let $M^n$ be a closed, odd-dimensional Riemannian manifold with positive sectional curvature and $T^r$ symmetry. Assume that the universal cover of $M$ is a homotopy sphere. If $r \geq \frac{n+1}{2q} + 1$ where $q$ is the smallest prime dividing $\frac{n+1}{2}$, then $\pi_1(M)$ is cyclic.
\end{theorem}

\begin{proof}
First, if $n$ is a dimension such that $q = 2$, then the theorem holds by Wilking's result (Theorem \ref{thm:WFRW}). We may assume therefore that $q \geq 3$. In other words, $2$ does not divide $\frac{n+1}{2}$, which means that $n \equiv 1 \bmod{4}$.  By the Davis--Weinberger factorization, $\pi_1(M) \cong \Z_{2^e} \times \Gamma$ for some odd-order group $\Gamma$.

Consider the free $\Gamma$--action on the universal cover of $M$. By Smith's theorem (Theorem \ref{thm:Smith44}), every abelian subgroup of $\Gamma$ is cyclic. By Burnside's classification (Theorem \ref{thm:Burnside}), $\Gamma \in \mathcal C_d$ for some odd $d \geq 1$. It suffices to prove that $d = 1$.

On one hand, Lemma \ref{lem:GroupCohomology} implies that $\Gamma$ has $\gcd(2d,n+1)$--periodic cohomology. By the same lemma, it suffices to show that this period is two. To show this, first note that Proposition \ref{pro:Room} implies $d \leq \frac{n+1}{2r} < q$. In particular, any prime divisor of $d$ will be less than $q$. By the definition of $q$, any such prime will not divide $\frac{n+1}{2}$. Hence $d$ is relatively prime to $\frac{n+1}{2}$, which implies that $\gcd(2d,n+1) = 2$, as required.
\end{proof}

We proceed to the proof of the other special case of the conjecture stated in the introduction. It is the second part of Theorem \ref{thm:SCC}, which we reproduce here.

\begin{theorem}\label{thm:SCCempty}
Let $q$ be a prime. Let $n$ be an odd, positive integer such that $q$ is the smallest prime dividing $\frac{n+1}{2}$. Let $M^n$ be a closed, positively curved manifold with $T^r$ symmetry such that $r \geq \frac{n+1}{2q} + 1$. If $n \geq 16q^2$ and if the fixed-point set of $T^r$ is empty, then $\pi_1(M)$ is cyclic.
\end{theorem}

For each $q$, this confirms the conjecture in all but finitely many dimensions under the additional assumption that the torus action has no fixed points. To illustrate how this assumption is used, we prove the following.

\begin{lemma}\label{lem:SCCempty}
Let $M^n$ be a closed, positively curved Riemannian manifold with $n \equiv 1 \bmod{4}$. If $M$ admits an isometric torus action of rank $r \geq 2\sqrt n$ such that there are no fixed points, then any $\Gamma \subseteq \pi_1(M)$ of odd order satisfies the following:
	\begin{enumerate}
	\item If $\Gamma \supseteq \Z_p \times \Z_p$ for some prime $p$, then $2p$ divides $n+1$.
	\item If $\Gamma \in \mathcal C_d$ for some $d \geq 1$, then $2d$ divides $n+1$.
	\end{enumerate}
\end{lemma}

Theorem \ref{thm:SCCempty} follows easily by combining Lemma \ref{lem:SCCempty} with Theorem \ref{thm:CCSlogPLUS}, so we prove it now.

\begin{proof}[Proof Theorem \ref{thm:SCCempty}]
The theorem holds at $q = 2$ by Wilking's result (Theorem \ref{thm:WFRW}), so we may assume $q \geq 3$. In particular, the divisibility assumption implies that $n \equiv 1 \bmod{4}$.

Since $n \geq 16 q^2$, the symmetry assumption implies
	\[r > \max\of{2\sqrt{n}, \log_{4/3}\pfrac{n+3}{6}}.\]
In particular, Theorem \ref{thm:CCSlogPLUS} and Lemma \ref{lem:SCCempty} apply. Together these results imply that $\pi_1(M) \cong \Z_{2^e} \times \Gamma$ for some $e \geq 0$ and some odd-order group $\Gamma$ such that the following hold:
	\begin{enumerate}
	\item If $\Gamma \supseteq \Z_p \times \Z_p$ for some prime $p$, then $p \leq \frac{n+1}{2r}$ and $2p $ divides $ n+1$.
	\item If $\Gamma \in \mathcal{C}_d$ for some $d\geq 1$, then $d \leq \frac{n+1}{2r}$ and $2d$ divides $ n+1$.
	\end{enumerate}
Since $r > \frac{n+1}{2q}$ and $q$ is the minimum prime dividing $\frac{n+1}{2}$, the first of these implies that every abelian subgroup of $\Gamma$ is cyclic. By Burnside's classification, $\Gamma \in \mathcal C_d$ for some $d \geq 1$. But now the second statement similarly implies that $d = 1$. By definition of $\mathcal C_d$, $\Gamma$ is cyclic, which implies that $\pi_1(M)$ is cyclic.
\end{proof}

We proceed to the proof of Lemma \ref{lem:SCCempty}, which occupies the rest of this section. First, if $M$ is a rational homology sphere, then the lemma follows immediately from Proposition \ref{pro:Room}. Assume therefore that $M$ is not a rational sphere.

Next, note that $2\sqrt{n} \geq \frac{n}{6} + 1$ for all $n \leq 81$, so $\Gamma$ is cyclic by the result of Frank--Rong--Wang (Theorem \ref{thm:WFRW}). Conclusion (1) is vacuously true, and Conclusion (2) is trivially true since $d = 1$ and $n \equiv 1 \bmod{2}$. 

Assume therefore that $n > 81$. In this range, Proposition \ref{pro:AKsqrt} implies that the fixed-point set of every involution in $T^r$ has codimension greater than $\sqrt{n}$.

As in the proof of Theorem \ref{thm:CCSlogPLUS}, we may choose a subtorus $T^{r-1} \subseteq T^r$ fixing a point $x \in M$ and a subgroup $\Z_2^{r-2} \subseteq T^{r-1}$ all of whose elements $\tau$ have $\cod\of{M^\tau_x} \equiv 0 \bmod{4}$. Here and throughout this section, the notation $M^\tau_x$ denotes the component of the fixed point set of $\tau$ that contains $x$. 

Choose $\iota_1 \in \Z_2^{r-2}$ such that $N_1 = M^{\iota_1}_x$ has minimal codimension. By minimality, at most a two-dimensional subtorus of $T^r$ fixes $N_1$. Moreover, since $r-2 > 2 \log_2(n)$ in this range of dimensions, $N_1$ has codimension less than $\frac{n}{2}$. Indeed, this follows from an estimate using error correcting codes (e.g., take $c = 2$ in \cite[Lemma 1.8]{Kennard14}).

Choose any $\Z_2^{r-3} \subseteq \Z_2^{r-2}$ that does not contain $\iota_1$, and choose further a subgroup $\Z_2^{r-4} \subseteq \Z_2^{r-3}$ all of whose involutions $\tau$ satisfy $\cod\of{(N_1)^{\tau}_x} \equiv 0 \bmod{4}$. Choose $\iota_2 \in \Z_2^{r-4}$ such that $\cod\of{(N_1)^{\iota_2}_x \subseteq N_1}$ is minimal. By minimality, at most a two-dimensional subtorus of the one acting on $N_1$ fixes $(N_1)^{\iota_2}_x$, and hence an at most four-dimensional subtorus of $T^r$ fixes $(N_1)^{\iota_2}_x$. Moreover, since $r - 2 > 2\log_2(\dim N_1)$, we again have that the codimension of $(N_1)^{\iota_2}_x \subseteq N_1$ is less than half of $\dim(N_1)$. Consider the isotropy at $x$, we see that we may replace $\iota_2$ by $\iota_1\iota_2$, if necessary, so that $N_2 = M^{\iota_2}_x$ has codimension less than $\frac{n}{2}$. By the connectedness lemma (see second remark following Theorem \ref{thm:Connectedness}), the intersection $N_1 \cap N_2$ is connected and equals $(N_1)^{\iota_2}_x$.

Finally, set $N_{12} = M^{\iota_1\iota_2}_x$. We make the following:

\begin{claim}
If $\Gamma$ contains a copy of $\Z_p\times\Z_p$ for some prime $p$, then $2p$ divides $\dim(N) + 1$ for every $N \in \{N_1, N_2, N_{12}, N_1\cap N_2\}$.
\end{claim}

\begin{proof}[Proof of claim]
First, we have already established that $\cod(N_i) < n/2$ for $i\in\{1,2\}$ and that $N_1 \cap N_2$ is connected. Consequently, Frankel's theorem implies that $N_i$ is the only component of $M^{\iota_i}$ with dimension greater than $\frac{n}{2}$. In particular, $\Gamma$ acts on each $N_i$ and hence on $N_1 \cap N_2$. Since $N_1 \cap N_2 = M^{\langle\iota_1,\iota_2\rangle}_x$ is contained in $N_{12} = M^{\iota_1\iota_2}_x$,  $\Gamma$ acts on $N_{12}$ as well.

Second, we have also established that $\cod(N_i) \geq \sqrt{n}$ for each $i \in \{1,2\}$. Moreover, since $N_1$ and $N_2$ have codimension less than $n/2$, Lemma \ref{lem:dk2} implies that each $N_i$ admits an isometric $T^{r-1}$ action commuting with the action of $\Gamma$. Since
	\[r - 1 \geq 2\sqrt{n} - 1 \geq 2\sqrt{n - \sqrt{n}} \geq 2\sqrt{\dim(N_i)},\]
and since $\cod(N_i) \equiv 0 \bmod{4}$, we have by induction that $2p$ divides $\dim(N_i) + 1$ for $i \in \{1,2\}$.

Third, if $\cod(N_{12}) < \frac{n}{2}$, then the argument from the previous paragraph applied to $N_{12}$ implies that $2p$ divides $ \dim(N_{12}) + 1$. Otherwise, $\cod(N_{12}) \geq \frac{n}{2}$. Since $N_1 \cap N_2$ admits $T^{r-4}$ symmetry commuting with the $\Gamma$--action, so does $N_{12}$. Since
	\[r - 4 \geq 2\sqrt{n} - 4 \geq 2\sqrt{n/2} \geq 2\sqrt{\dim(N_{12})},\]
for $n > 81$, we have by induction that $2p $ divides $ \dim(N_{12}) + 1$.

Finally, we claim that $2p $ divides $ \dim(N_1 \cap N_2) + 1$. If some two-dimensional torus acting on $N_1$ fixes $N_1 \cap N_2 = (N_1)^{\iota_2}_x$, then $N_1$ is a rational homology sphere by Lemma \ref{lem:dk2}. It follows then by Proposition \ref{pro:CircleObstruction} that
	\[\dim(N_1 \cap N_2) + 1 \equiv \dim(N_1) + 1 \equiv 0 \bmod{2p},\]
as required. If there is no such two-dimensional torus, then $N_1 \cap N_2$ admits $T^{r-2}$ symmetry. By an estimate similar to those above, we have $r - 2 \geq 2 \sqrt{\dim(N_1 \cap N_2)}$. Moreover, since we chose $\iota_1$ and $\iota_2$ so that
	\[\dim(N_1 \cap N_2) \equiv \dim(N_1) \equiv 1 \bmod{4}\]
and so that $\Gamma$ acts on $N_1 \cap N_2$, we have by induction that $2p$ divides $\dim(N_1 \cap N_2) + 1$.
\end{proof}

Given the claim, we conclude that $2p$ divides $n + 1$ from the following observation:
	\[n  + 2\dim(N_1\cap N_2) = \dim(N_1) + \dim(N_2) + \dim(N_{12}).\]
This completes the proof of Conclusion (1) of Lemma \ref{lem:SCCempty}. The proof of Conclusion (2) is similar, so the proof of Lemma \ref{lem:SCCempty} and hence of Theorem \ref{thm:SCCempty} is complete.

%%%%% Bibliography %%%%%
%\bibliographystyle{plain}%{alpha}
%\bibliography{myrefs}

\end{document}